\documentclass{article}
\usepackage{amsmath, amsfonts, cancel, mathabx, bbm, graphicx, subcaption, stmaryrd, tikz-cd}
\usepackage[margin=1in]{geometry}
\usepackage[square,numbers, sort&compress]{natbib}
\usepackage{amsthm, thmtools}

\usepackage[colorlinks, allcolors=blue]{hyperref}
\usepackage[noabbrev, nameinlink]{cleveref}

\DeclareMathAlphabet{\mymathbb}{U}{BOONDOX-ds}{m}{n}

\makeatletter
\def\smalloverbrace#1{\mathop{\vbox{\m@th\ialign{##\crcr\noalign{\kern3\p@}%
\tiny\downbracefill\crcr\noalign{\kern3\p@\nointerlineskip}%
$\hfil\displaystyle{#1}\hfil$\crcr}}}\limits}
\makeatother

\newcommand{\downset}[2][ ]{%
\downarrow \! \{ #2 \}_{#1}
}

\newcommand{\downsetdisplay}[2][ ]{%
\downarrow \!\!\! \{ #2 \}_{#1}
}

\newcommand{\upset}[2][ ]{%
\uparrow \! \{ #2 \}_{#1}
}

\newcommand{\upsetdisplay}[2][ ]{%
\uparrow \!\! \{ #2 \}_{#1}
}

\DeclareSymbolFont{extraup}{U}{zavm}{m}{n}
\DeclareMathSymbol{\varspade}{\mathalpha}{extraup}{85}
\DeclareMathSymbol{\varheart}{\mathalpha}{extraup}{86}
\DeclareMathSymbol{\vardiamond}{\mathalpha}{extraup}{87}
\DeclareMathSymbol{\varclub}{\mathalpha}{extraup}{88}

\newcommand{\unsim}{\mathord{\sim}}

\newcommand{\lc}{\left\{}          
\newcommand{\rc}{\right\}}         
\newcommand{\lav}{\left\vert}      
\newcommand{\rav}{\right\vert}     
\newcommand{\ev}[1]{( #1 )} 
\newcommand{\ew}[1]{\left( #1 \right)}

\newcommand{\Z}{\mathbb{Z}}
\newcommand{\C}{\mathbb{C}}

\newtheorem{theorem}{Theorem}[section]
\newtheorem{lemma}[theorem]{Lemma}
\newtheorem{corollary}[theorem]{Corollary}
\newtheorem{remark}{Remark}

\numberwithin{equation}{section}
\theoremstyle{definition}
\newtheorem{definition}{Definition}
\newtheorem{proposition}[theorem]{Proposition}
\newtheorem{example}{Example}

\title{Partitions of an Eulerian Digraph into Circuits}

\usepackage{authblk}

\author[1]{Joshua Cooper}
\author[2]{Utku Okur}
\affil[1]{Department of Mathematics, University of South Carolina, U.S.A. (\href{mailto:email}{cooper@math.sc.edu})}
\affil[2]{Department of Mathematics, University of South Carolina, U.S.A. (\href{mailto:email}{uokur@email.sc.edu})}

\date{\today}

\newcommand{\circled}[1]{
\tikz[baseline=(char.base)]{
\node[shape=circle, draw, inner sep=1.5pt, font=\normalsize ] (char) {#1};
}
}

\begin{document}
\maketitle

\begin{abstract}
We investigate a cancellation property satisfied by a connected Eulerian digraph $D$. Namely, unless $D$ is a single directed cycle, we have $\sum_{k\geq 1} (-1)^{k} f_k\ev{D} =0 $, where $f_k\ev{D}$ is the number of partitions of Eulerian circuits of $D$ into $k$ circuits. This property is a consequence of the fact that the Martin polynomial of a digraph has no constant term. We provide an alternative proof by employing Viennot's theory of Heaps of Pieces, and in particular, a bijection between closed trails of a digraph and heaps with a unique maximal piece, which are also in bijection with unique sink orientations of the intersection graphs $G_a$ of partitions $a$ of $E(D)$ into cycles. The argument considers the partition lattice of the edge set of a digraph $D$, restricted to the join-semilattice $T(D)$ induced by elements whose blocks are connected and Eulerian. The minimal elements of $T\ev{D}$ are exactly the partitions of $D$ into cycles, and the up-set of a minimal element $a\in T(D)$ is shown to be isomorphic to the bond lattice $L(G_a)$. Using tools developed by Whitney and Rota, we perform M\"{o}bius inversion on $T(D)$ and obtain the claimed cancellation.

As a consequence of this alternative proof, we relate the Martin polynomial of a digraph directly to the chromatic polynomials of the intersection graphs of partitions of $D$ into cycles. Finally, we apply the cancellation property in order to deduce the classical Harary-Sachs Theorem for graphs of rank $2$ from a hypergraph generalization thereof, remedying a gap in a previous proof of this.


\end{abstract}

\hypersetup{linkcolor=black} 

\tableofcontents

\hypersetup{linkcolor=blue}

\section{Introduction}
The present work is motivated by a curious cancellation property that is claimed in \cite{cc2}:
\begin{equation}
\label{eqn:cancellation}
\sum_{k\geq 1} \ev{-1}^{k} f_k\ev{D} = 0
\end{equation}
for any Eulerian digraph $D$ which is not a cycle, where $f_k\ev{D}$ is the number of partitions of (the edge set of) Eulerian circuits of $D$ into $k\geq 1$ circuits.

One way of deriving (\ref{eqn:cancellation}) is the evaluation at $t=0$ of the Martin polynomial, the generating function in $t-1$ of partitions into circuits (q.v. \Cref{def:martin_circuit_partition}). As this polynomial has no constant term by \cite[Theorem 3.2, p.~5]{lasvergnas}, the cancellation property in question follows. 

In \Cref{sec:bond_lattice_mobius}, we provide an alternative proof of (\ref{eqn:cancellation}), by connecting some machinery developed in the literature. For a graph $G$, and a fixed vertex $x$, let $\overline{\mathcal{O}}^{x}\ev{G}$ be the set of acyclic orientations of $G$ with unique sink $x$. In \cite[p.~325]{viennot}, the theory of Heaps of Pieces is developed, and it is noted that acyclic orientations of $G$ with unique sink $x$ are in bijection with full pyramids with maximal piece $x$, where the pieces are vertices of $G$ and concurrence relation is sharing an edge. On the other hand, the equinumerosity of the acyclic orientations $\overline{\mathcal{O}}^{x}\ev{G}$ with unique sink $x$ (independent of the choice of the sink $x$) and of NBC bases $\mathcal{N}\ev{G}$ of the bond lattice $L\ev{G}$ (terminology defined below) follows by combining \cite{zaslavsky} and \cite{whitney}. In fact, explicit bijections between NBC bases of $\mymathbb{1}_{L\ev{G}}$ and acyclic orientations with unique sink were given in \cite{benson} and \cite{sagan_gebhard}. In \Cref{sec:bond_lattice_mobius}, we give a summary of the construction of the bijection $\mu: \mathcal{N}\ev{ G } \rightarrow \overline{\mathcal{O}}^{x}\ev{G} $ defined in \cite[Theorem 4.2, p.~11]{benson} (\Cref{thm:NBC_bij_explicit} below). Afterwards, we recursively define bijections $\begin{tikzcd}
\mathcal{N}\ev{ G } \arrow[r,shift left=2pt,"\phi"] & \arrow[l,shift left=2pt,"\psi"] \mymathbb{P}^{x}\ev{ G } 
\end{tikzcd}$ in the language of Heaps of Pieces, which yield bijections:
$$
\begin{tikzcd}
\mathcal{N}\ev{ G } \arrow[r,shift left=2pt,"\phi'"] & \arrow[l,shift left=2pt,"\psi'"] \overline{\mathcal{O}}^{x}\ev{ G } 
\end{tikzcd}
$$
It turns out that $\phi'$ and $\psi'$ are alternative (recursive) definitions of $\mu$ and $\pi$ of \cite{benson}. We show in \Cref{sec:bond_lattice_mobius} Equation (\ref{eqn:equal_bij}), that $\phi'$ is equal to $\mu$, which implies that $\pi = \psi'$. 

Further, we can find a bijective mapping between NBC bases and critical configurations of a graph in \cite{biggs_winkler}, where the chromatic polynomials of ``wheel graphs" are calculated, as an application. A closely related result is that the total number $|\overline{\mathcal{O}}\ev{G}|$ of acyclic orientations of $G$ is the coefficient $[t] P_{G}\ev{ t } $ of the chromatic polynomial $P_{G}\ev{ t }$ of $G$, which is also given by the number $|\mathcal{N}^{*}\ev{ G }|$ of NBC sets. In fact, a bijection $\overline{\mathcal{O}}\ev{G}\rightarrow \mathcal{N}^{*}\ev{ G }$ between acyclic orientations and NBC sets is given in \cite{sagan_blass}. It is also known that the number $|\overline{\mathcal{O}}^{x}\ev{G}|$ of acyclic orientations of $G$ with unique sink $x$ is given by $|P_{G}\ev{ -1 }|$ (c.f. \cite{stanley_2}).

The celebrated NBC theory finds it root in the work of Whitney (\cite{whitney}), which calculates the chromatic polynomial of a graph $G$, by the counts of the ``no broken circuit (NBC) sets" of $G$ (q.v.~\Cref{def:NBC_sets} below). A generalization of Whitney's theorem is due to Rota (\cite[Proposition 1, p.~358]{rota}), the NBC Theorem (q.v.~\Cref{thm:rota_NBC_thm} below), which calculates the M\"{o}bius function of a finite geometric lattice $L$, via the NBC sets of $L$ and implies Whitney's Theorem, when $L = L(G)$ is the ``bond lattice" of $G$ (q.v.~\Cref{def:bond_lattice_graph} below). The NBC Theorem was further generalized to the NBB Theorem (\cite[Theorem 1.1, p.~2]{sagan}) by Blass and Sagan. We can find an application of the NBC and NBB Theorems in the calculation of the M\"{o}bius function of the ``noncrossing bond lattice" in \cite{hallam}. See \cite{stanley, sagan, rota, whitney, bernardi, benson, zaslavsky, bernardi_nadeau} for other numerous and general connections between the chromatic polynomial, Tutte polynomial, NBC bases and acyclic orientations with unique sink, for a graph $G$. Also, see \cite{fu_peng_wang_yang} for a bijection between acyclic orientations and NBC sets in the more general setting of a matroid. 

In \Cref{sec:partition_lattice_eulerian_digraph}, we consider the partition lattice of the edge-set $E\ev{D}$ of an Eulerian digraph $D$. We define the join-semilattice $T\ev{D}$ induced by partitions of $E\ev{D}$ into Eulerian sub-digraphs, where the minimal elements are exactly partitions of $E\ev{D}$ into cycles. By an application of the theory outlined in \Cref{sec:bond_lattice_mobius}, we obtain an alternative proof of (\ref{eqn:cancellation}). The alternative approach we take allows us to derive an identity which expresses $t$ times the Martin polynomial of a digraph as a signed sum over the chromatic polynomials of the intersection graphs of partitions of $D$ into cycles.

The cancellation property (\ref{eqn:cancellation}) of a digraph $D$ is independently interesting, and here it allows us to deduce the classical Harary-Sachs Theorem for graphs from the more general version for hypergraphs given in \cite{cc1} and \cite{co2}. The Harary-Sachs Theorem for hypergraphs expresses the coefficients of the characteristic polynomial $\chi_{\mathcal{H}}\ev{t}$ of a hypergraph of rank $k\geq 2$ (i.e., each edge has $k$ vertices), in terms of its constitutive \textit{infragraphs}, namely, multi-hypergraphs where the degree of each vertex is divisible by $k$ (q.v.~\Cref{def:multi_hyper_graph} and \Cref{def:veblen_multi_hypergraph} below). A connected infragraph is \textit{decomposable} if it contains a proper non-trivial sub-infragraph, and \textit{indecomposable} otherwise. When we specialize to rank $k=2$, the contribution of any infragraph $X$ to the coefficients of the characteristic polynomial, (called $w_{|V\ev{\mathcal{H}}|}\ev{X}$ in \Cref{thm:harary_sachs_hyper}), is zero, unless $X$ is \textit{elementary}, in other words, each component of $X$ is indecomposable, i.e., a cycle of length $\geq 3$ or an edge. This was claimed in \cite{cc2}, where (\ref{eqn:cancellation}) is erroneously deduced as an easy application of the Inclusion-Exclusion principle. In \Cref{sec:partition_lattice_eulerian_digraph}, we correct this argument by providing a full proof and in \Cref{sec:deduction}, we deduce the Harary-Sachs Theorem for graphs from the Harary-Sachs Theorem for hypergraphs, by an application of (\ref{eqn:cancellation}). 

It should be noted that when we consider hypergraphs of rank $k=3$, the contribution of a decomposable infragraph to the coefficients of the characteristic polynomial can be non-zero. Regardless, there appears to be a cancellation property for hypergraphs as well. Clearly, just because the characteristic polynomial $\chi_\mathcal{H}\ev{t}$ is a \textit{polynomial}, it follows that, for any $d\geq N$, where $N$ is the degree of $\chi_\mathcal{H}\ev{t}$, the sum of the contributions of all infragraphs with $d$ edges, to the codegree $d$th term of $\chi_\mathcal{H}\ev{t}$ is zero. However, a stronger phenomenon is noted in \cite[Remark 14, p.~30]{co2}: We have $ w\ev{X} = 0 $, for any infragraph $X$ such that $|E\ev{X}| > |V\ev{X}| \ev{ k-1 }^{|V\ev{X}|-1}$. Even though we have an algebraic confirmation of this fact, via a particular use of a theorem on resultants (\cite[Theorem 2.3, p.~86]{usingalg}), a combinatorial proof as in the case of graphs would be interesting, as it could lead to a better understanding of the algebraic multiplicity of zero in the characteristic polynomial, i.e., the highest power of $t$ dividing $\chi_{\mathcal{H}}(t)$, which is directly tied to the last non-zero contribution by infragraphs.

\section{The Bond Lattice and Heaps of Pieces}
\label{sec:bond_lattice_mobius}

\subsection{Preliminaries}
In this paper, we use $A\sqcup B$ to denote the union of two disjoint sets $A,B$. The \textit{cardinality} of a finite set $A$ is denoted by $|A|$. For a finite set $A$ and an integer $0\leq k\leq |A|$, the $k$-subsets of $A$ is denoted $\binom{ A }{k}$. 

We start with the definition of a multi-(hyper)graph of rank $k\geq 2$ and of a multi-digraph, following the notation of \cite[p.2]{bondy}. Specializing to rank $k=2$ yields multi-graphs. 
\begin{definition}[Multi-hypergraphs and parallel edges]
\label{def:multi_hyper_graph}
A \textit{multi-hypergraph} $X$ of rank $k\geq 2$ is a triple $X = \ev{ V\ev{X}, E\ev{X}, \phi }$, where $V\ev{X}, E\ev{X}$ are finite sets, and $\phi: E \rightarrow \binom{V}{k}$ is a mapping, such that the coordinates of $\phi\ev{e}$ are pairwise distinct, for each $e\in E\ev{X}$ (there are no degenerate edges or loops.) The multi-hypergraph $X$ is called a \textit{multi-graph} if it is of rank $2$. 
\begin{enumerate}
\item If $X = \ev{ V\ev{X}, E\ev{X}, \phi }$ is given such that $\phi$ is injective, then $X$ is a \textit{simple} hypergraph. We omit $\phi$, whenever $X$ is a simple graph. We also sometimes omit ``simple" and write ``graph", to mean a simple graph. 
\item Two edges $e_1, e_2 \in E\ev{X}$ are \textit{parallel}, denoted $e_1 \sim e_2$, provided $\phi\ev{e_1} = \phi\ev{e_2}$. Given a multi-hypergraph $X$ and an edge $e\in E\ev{X}$, then the \textit{multiplicity} $m_X\ev{e}$ of $e$ is the cardinality $| [e]_{\unsim} | = | \{ f\in E\ev{X}: \phi\ev{f} = \phi\ev{e} \} | $ of the equivalence class of edges parallel to $e$. For each $v_1,\ldots,v_k \in V\ev{X}$, we write $m_X\ev{ \{ v_1, \ldots, v_k \} } = | \{ e\in E\ev{X} : \phi\ev{e} = \{ v_1, \ldots, v_k \} \} | $. The \textit{degree} of a vertex $u$ is defined as, $\textup{deg}_X\ev{u} = |\{ e\in E\ev{X}: u\in e\} | $. 
\item For a simple graph $G$ and a subset $S\subseteq V\ev{G}$ of vertices, we write $G[S]$ for the subgraph induced by $S$.
\end{enumerate}
\end{definition}

Next, we define multi-directed graphs, or shortly, digraphs: 
\begin{definition}[Multi-digraphs and orientations of a multi-graphs]
\label{def:orientations}
A \textit{multi-digraph} $D=\ev{ V\ev{D}, E\ev{D}, \psi}$ of rank $2$ is a triple, where $V\ev{D}, E\ev{D}$ are finite sets, and $\psi: E\ev{D} \rightarrow V\ev{D} \times V\ev{D}$ is a function, such that the coordinates of $\psi\ev{e}$ are distinct, for each edge $e\in E\ev{D}$. (There are no loops.) Sometimes, we shortly write ``digraph", instead of multi-digraph.
\begin{enumerate}
\item Two edges $e_1, e_2 \in E\ev{D}$ are \textit{parallel}, denoted $e_1 \unsim e_2$, provided $\psi\ev{e_1} = \psi\ev{e_2}$. (We use the same notation for the equivalence relation of parallelism in a digraph.) 
\item Given a digraph $D$ and an edge $e\in E\ev{D}$, then the \textit{multiplicity} $m_D\ev{e}$ of $e$ is the cardinality of the set $[e]_{\unsim} = \{ f\in E\ev{D}: \psi\ev{f} = \psi\ev{e} \}$ of edges parallel to $e$. Also, we define $m_D\ev{u,v} = | \{ e\in E\ev{D} : \psi\ev{e} = \ev{u,v} \}|$, for each $u,v\in V\ev{D}$.
\item The \textit{in-degree} of a vertex $u$ is defined as $\textup{deg}_D^{-}\ev{u} = \sum_{v\in V\ev{D}} m\ev{ v,u }$. The \textit{out-degree} of $u$ is $\textup{deg}_D^{+}\ev{u}:=\sum_{v\in V\ev{D}} m\ev{ u,v }$. In particular, we have $\textup{deg}_D\ev{u} = \textup{deg}_D^{-}\ev{u} + \textup{deg}_D^{+}\ev{u} $.
\item Given an undirected multi-graph $X=(V(X),E(X),\phi)$ of rank $2$ and a digraph $O=(V(O),E(O),\psi)$, then $O$ is an \textit{orientation} of $X$, provided $V\ev{O} = V\ev{X}$ and $E\ev{O} = E\ev{X}$ and we have $\phi\ev{ e } = \theta\ev{ \psi\ev{ e } }$ for each $e\in E\ev{X} = E\ev{O}$, where 
\begin{align*}
\theta: V\ev{X} \times V\ev{X} &\rightarrow \binom{V\ev{X}}{2} \\ 
\ev{ v_1, v_2} & \mapsto \{ v_1, v_2\}
\end{align*}
is the "forgetful" mapping. The set of orientations of $X$ is denoted as $\mathcal{O}\ev{X}$. So, we have $|\mathcal{O}\ev{X}| = 2^{|E\ev{X}|}$. 
\end{enumerate}
\end{definition}

\begin{definition}
\label{def:unique_sink_orientations}
For a digraph $D$ and a vertex $x\in V\ev{D}$, then $x$ is a \textit{sink} of $D$, provided $\textup{deg}^{+}_D\ev{x} = 0$. We say that $x$ is the \textit{unique sink} of $D$, if $\{ y\in V\ev{D} \setminus \{ x \} : \textup{deg}_D^{+}\ev{y} = 0 \} = \emptyset$. For a graph $G$ and a vertex $x\in V\ev{G}$, we write $\overline{\mathcal{O}}^{x}\ev{G}$ for the set of orientations of $G$ with unique sink $x$. By definition, we have
$$
\bigsqcup_{x\in V\ev{G}} \overline{\mathcal{O}}^{x}\ev{G} \subseteq \mathcal{O}\ev{G}
$$
\end{definition}

Next, we have some preliminaries about posets and lattices. 
\begin{definition}
\label{def:posets_prelim}
Let $\ev{ \Omega, \leq }$ be a poset. 
\begin{enumerate}
\item The \textit{down-set} of an element $a$ is defined as, $ \downset[\Omega]{a} := \{ b\in \Omega: b \leq a \}$. The \textit{up-set} of $a$ is the set $ \upset[\Omega]{a} := \{ b\in \Omega: b \geq a \}$. 
\item $\max\ev{ \Omega }$ denotes the maximal elements of $\Omega$. If the maximal element is unique, then it is denoted $\mathbbm{1}_\Omega$. $\min\ev{\Omega}$ denotes the minimal elements of $\Omega$. If the minimal element is unique, then it is denoted $\mymathbb{0}_\Omega$. 	
\item Given a poset $\Omega$ and two elements $x,y\in \Omega$, then $y$ \textit{covers} $x$, denoted $x\lessdot y$, provided $x < y$ and there is no $z\in \Omega$ such that $x<z<y$.
\end{enumerate}
Assuming $\Omega$ has a unique minimal element $\mymathbb{0}_\Omega$, 
\begin{enumerate}
\item A function $\textup{rk}: \Omega\rightarrow \Z^{\geq 0} = \{0,1,2,\ldots,\} $, is a \textit{rank function}, provided $\textup{rk}\ev{\mymathbb{0}_\Omega}=0$ and $x \lessdot y $ if and only if $\textup{rk}\ev{x} + 1 = \textup{rk}\ev{y}$, for each $x,y \in \Omega$. A poset with a unique minimal element is a \textit{ranked poset}, if it has a rank function. By induction on the number of elements, it can be shown that the rank function is unique, provided it exists. 
\item The \textit{rank} of $\Omega$ is defined as $\textup{rk}\ev{\Omega} := \max_{x\in \Omega} \textup{rk}\ev{x}$. The \textit{characteristic polynomial} of a ranked poset is $\sum_{x\in \Omega} \mu_\Omega\ev{\mymathbb{0},x} t^{\textup{rk}\ev{\Omega} - \textup{rk}\ev{x}}$ (c.f. \cite{sagan_hallam})
\end{enumerate}
\end{definition}
We omit the subscripts of $\downset[\Omega]{a}, \upset[\Omega]{a}, \mymathbb{1}_\Omega$, $\mymathbb{0}_\Omega$, when it is clear from the context. As defined in \cite[p.~305]{vanlint}, a poset $L$ is a lattice if every finite subset $S\subseteq L$ has a meet and a join. Note that a finite lattice $L$ necessarily has a unique minimal element $\mymathbb{0}$ and a unique maximal element $\mathbbm{1}$. Furthermore, an element $p\in L$ is a \textit{point} or \textit{atom}, if $\mymathbb{0} \lessdot p$. An element $h\in L$ is a \textit{copoint}, if $h\lessdot \mathbbm{1}$. 

\begin{definition}

\phantom{a}

\begin{enumerate}
\item Let $\ev{ \Omega_1, \leq_1}$ and $\ev{ \Omega_2, \leq_2}$ be two posets. The \textit{product} of $\Omega_1$ and $\Omega_2$, denoted $\Omega_1 \times \Omega_2$, is the poset $\ev{ \Omega_1\times \Omega_2, \leq}$, where $\ev{x_1,x_2}\leq \ev{y_1,y_2}$ if and only if $x_1\leq_1 y_1$ and $x_2\leq_2 y_2$, for $x_1,y_1\in \Omega_1$ and $x_2,y_2\in \Omega_2$.
\item Given a poset $\ev{\Omega,\leq}$, then the \textit{dual} of $\Omega$ is the poset $ \Omega^{d} = \ev{\Omega, \leq_d}$, where $a\leq_d b $ if and only if $b\leq a$, for each $a,b\in \Omega$.
\end{enumerate}
\end{definition}
It is easy to see that the M\"{o}bius matrix of the dual poset $\Omega^d$ is the transpose of the M\"{o}bius matrix of $\Omega$. In other words, 
$$ \mu_{\Omega^d}\ev{ a,b } = \mu_{\Omega}\ev{ b,a }$$
for each $a,b \in \Omega$. 

\begin{lemma}
\label{lem:mu_down_set}
Let $\ev{\Omega,\leq}$ be a finite ranked poset. Then, we have
$$ \mu_\Omega\ev{ a, b } = \mu_{ \downset[\Omega]{b} }\ev{ a, b } $$
for each $a,b\in \Omega$.
\end{lemma}
\begin{proof}
If $a\not \leq b$, then the statement vacuously holds, as both sides of the equation are equal to zero. Assuming $a\in \downset[\Omega]{b}$, we show the statement by (descending) induction on the rank of $a$. The image of $\downset[\Omega]{b}$ under the rank function $\textup{rk}$ is a closed interval of integers. In other words, $\{ \textup{rk}\ev{ c }: c \in \downsetdisplay[\Omega]{b} \} = [0 , w]$, where $\textup{rk}\ev{ b } = w$ is the maximum rank attained. For the base case, let $\textup{rk}\ev{ a } = w$. Then, we have $a=b$ and so, $ \mu_\Omega\ev{ a, b } = 1 =  \mu_{ \downset[\Omega]{b} }\ev{ a, b }$. For the inductive hypothesis, we fix $u\in [1,w]$ and assume that the statement holds for all $c\in \downset[\Omega]{a}$ with $u \leq \textup{rk}\ev{ c } \leq w$. For the inductive step, we fix $a\in \downset[\Omega]{b}$ such that $\textup{rk}\ev{ a } = u - 1 \geq 0 $. Then, 
\begin{align*}
&\mu_\Omega\ev{ a,b } = -\sum_{ c\in \Omega: \ a < c \leq b } \mu_\Omega\ev{ c, b } 	\\
& = -\sum_{ c\in \downset[\Omega]{b} : \ a < c \leq b } \mu_\Omega\ev{ c, b } \\
& = -\sum_{ c\in \downset[\Omega]{b} : \ a < c \leq b } \mu_{ \downset[\Omega]{b} }\ev{ c, b } \\
& \hspace{3cm}\text{ by the inductive hypothesis, since $ \textup{rk}\ev{a} < \textup{rk}\ev{c} $ for each $c\in \Omega$ with $a<c\leq b$.} \\ 
& = \mu_{ \downset[\Omega]{b} }\ev{ a, b }
\end{align*}
\end{proof}

\begin{lemma}
\label{lem:mu_up_set}
Let $\ev{ \Omega,\leq }$ be a finite ranked poset. Then, 
$$\mu_\Omega\ev{ a, b } = \mu_{ \upset[\Omega]{a} }\ev{ a, b } $$
for each $a, b \in \Omega$. 
\end{lemma}
\begin{proof}
Consider the dual poset $\Omega^d$. For each $a\in \Omega$, we have $\upset[\Omega]{a} =  \downset[\Omega^d]{a} $, so we have, by \Cref{lem:mu_down_set}, 
$$ \mu_{ \upset[\Omega]{a} } \ev{ a, b } = \mu_{ \downset[\Omega^d]{a} } \ev{ b,a } = \mu_{ \Omega^d } \ev{ b,a } = \mu_{\Omega}\ev{a,b}$$
\end{proof}

\begin{definition}[Partition Lattice]
Let $\mathcal{A}$ be a finite set. We define the partition lattice $\Pi\ev{\mathcal{A}} = \ev{ \pi\ev{ \mathcal{A} }, \leq }$ on the elements of $\mathcal{A}$, as follows:
\begin{enumerate}
\item $\pi\ev{ \mathcal{A} } = \{ \{ A_1, \ldots, A_t \} : A_i\subseteq \mathcal{A} \text{ for each $i$ and }\bigsqcup_{i=1}^{t} A_i = \mathcal{A} \}$. 
\item Given an element $a = \{ A_1, \ldots, A_t \} \in \Pi\ev{ \mathcal{A} }$, the subsets $A_i$ are called \textit{parts} of $a$. 
\item Given $a,b\in \Pi\ev{ \mathcal{A} } $, then $a$ is a \textit{refinement} of $b$ (denoted $a\leq b$) if and only if each part of $b$ is a union of some parts of $a$. 
\item Let $|a| = k $ be the number of parts in a partition $a=\{A_1,\ldots,A_k\}$. 
\end{enumerate}
We define the lattice structure on the partition poset $\Pi\ev{\mathcal{A}}$, which we call the partition lattice hereafter. We define a concurrence relation on $\mathcal{P}\ev{\mathcal{A}}$: For each $C_1,C_2 \subseteq \mathcal{A}$, let $C_1 \varspade C_2 $ if and only if $C_1 \cap C_2 \neq \emptyset$. Let $\overline{\varspade}$ be the equivalence relation, obtained by taking the transitive closure of $\varspade$. 

Given two elements $a = \{ A_1, \ldots, A_t \}, b = \{ B_1,\ldots, B_r\} \in \Pi\ev{ \mathcal{A} }$, then, 
\begin{align*}
& a \wedge b = \{ A \cap B : A\in a, \ B\in b,\  A \varspade B\} \\
& a \vee b = \left\{ \bigcup_{C \in \mathcal{C} } C  \ : \  \mathcal{C} \in (a\cup b) / \overline{\varspade} \right\}
\end{align*}
The operations $\vee$ and $\wedge$ are associative, commutative and idempotent. In particular, $\bigwedge_{i=1}^{r} a_i$ and $\bigvee_{i=1}^{r} a_i$ are defined inductively, for any family $\{a_i\}_{i=1}^{k} \subseteq \Pi\ev{\mathcal{A}}$, which makes $\Pi\ev{\mathcal{A}}$ a complete lattice (\cite[Definition 4.1, p.~17]{burris}).
\end{definition}

\subsection{Heaps of Pieces Framework}
In this section, we will give a short summary of the theory of Heap of Pieces from \cite[Definition 2.1, p.~324]{viennot}. Let $\mathcal{B}$ be a finite set of pieces, and let $\mathcal{R}$ be a concurrence relation on $\mathcal{B}$, i.e., a reflexive and symmetric relation. 

Given a finite graph, we can obtain a set of pieces $\mathcal{B}_G := V\ev{G}$, with concurrence relation $\mathcal{R}_G = \{ \ev{u,v}: \{u,v\}\in E\ev{G} \text{ or } u=v\}$. Conversely, for a given set of pieces $\mathcal{B}$ with a concurrence relation $\mathcal{R}$, we can obtain a simple loopless graph $G_{\mathcal{B},\mathcal{R}}$ with $V\ev{G_{\mathcal{B},\mathcal{R}}} = \mathcal{B}$ and $E\ev{ G_{\mathcal{B},\mathcal{R}} } = \{ \{u,v\}: u\mathcal{R} v \text{ and } u\neq v\}$. Hence, a simple graph $G = \ev{ V\ev{G}, E\ev{G} } $ is identified with a pair $\ev{\mathcal{B},\mathcal{R}}$ of pieces and concurrence relation. 

\begin{definition}
\label{def:heap_viennot}
Let $\mathcal{B}$ be a set (of pieces) with a symmetric and reflexive binary (concurrence) relation $\mathcal{R}$. Then, a \textit{heap} is a triple
$$ H = \ev{ \Omega , \leq, \ell }$$
where $\ev{ \Omega , \leq }$ is a poset, $\ell: \Omega \rightarrow \mathcal{B}$ is a function ($\ell$ labels the elements of $\Omega$ by elements of $\mathcal{B}$) such that for all $x,y\in \Omega$,
\begin{enumerate}
\item[(i)] If $\ell\ev{x} \mathcal{R} \ell\ev{y}$ then $x\leq y$ or $y \leq x$.
\item[(ii)] If $x \lessdot y$, then $\ell\ev{x} \mathcal{R} \ell\ev{y}$.
\end{enumerate} 
\end{definition}

\begin{remark}
\label{rmk:equiv_heap}
Let $H = \ev{ \Omega, \leq }$ be a poset. Define the following associated graphs:
\begin{enumerate}
\item Let $\mathcal{H}\ev{H}$ be the \textup{Hasse graph} of $H$, defined as the pair $\ev{\Omega, \{ \{ x,y\}: x\lessdot y\}}$.
\item Let $\mathcal{C}\ev{H}$ be the \textup{comparability graph} of $H$, namely, $\ev{\Omega, \{ \{ x,y\}: x\leq y\}}$. 
\end{enumerate}
Given a set of pieces $\mathcal{B}$ and a concurrence relation $\mathcal{R}$, consider a labeling function $\ell: \Omega \rightarrow \mathcal{B}$. Then, condition (ii) of \Cref{def:heap_viennot} is equivalent to $\ell$ being a graph homomorphism $\mathcal{H}\ev{H} \rightarrow G_{\mathcal{B},\mathcal{R}}$. Condition (i) is equivalent to $\ell: \Omega \rightarrow \mathcal{B}$ being a graph homomorphism $\mathcal{C}\ev{H}^{c} \rightarrow G_{\mathcal{B},\mathcal{R}}^{c}$, where $G^{c} = \ev{ V\ev{G}, \binom{V\ev{G}}{2} \setminus E\ev{ G }}$ is the \textup{complement graph}. 
\end{remark}
In short, given $\mathcal{B}$ and $\mathcal{R}$, then a poset $H = \ev{\mathcal{B}, \leq, \ell}$ is a heap if and only if $G_{\mathcal{B},\mathcal{R}}$ is ``sandwiched" between the Hasse graph and comparability graph of $H$.

Next, we introduce some notation for heaps of pieces and classify heaps in terms of their labeling functions and the number of maximal pieces: (See \cite{co2}, for more detail.)
\begin{enumerate}
\item $\mathbbm{h}\ev{ \mathcal{B},\mathcal{R}}$ is the set of all \textit{heaps} with labels from $\mathcal{B}$ with concurrence relation $\mathcal{R}$.
\item $\mathbbm{p}\ev{ \mathcal{B},\mathcal{R}}$ denotes the set of \textit{pyramids}, i.e., heaps with a unique maximal piece. 
\item $\mymathbb{H}\ev{ \mathcal{B},\mathcal{R}}$ comprises \textit{full} heaps, i.e., heaps with a bijective labeling function $\ell: \Omega \rightarrow \mathcal{B}$.
\item $\mymathbb{P}\ev{ \mathcal{B},\mathcal{R}}$ denotes full pyramids. 
\item $\mymathbb{P}^{\beta}\ev{ \mathcal{B},\mathcal{R}}$ is the set of full pyramids with maximal piece $\beta$, for a given piece $\beta\in \mathcal{B}$.
\end{enumerate}

For the rest of the text, the labeling function will be omitted whenever we encounter full heaps. In particular, a poset $\ev{\Omega,\leq}$ is a full heap, provided for each $x,y\in \Omega$,
\begin{enumerate}
\item[(i)] If $x \mathcal{R} y$ then $x\leq y$ or $y \leq x$. 
\item[(ii)] If $x \lessdot y$, then $x \mathcal{R} y$.
\end{enumerate}
The first condition is equivalent to $G_{\mathcal{B},\mathcal{R}}$ being a subgraph of $\mathcal{C}\ev{H}$, and the second condition is equivalent to $\mathcal{H}\ev{H}$ being a subgraph of $G_{\mathcal{B},\mathcal{R}}$ (q.v.~\Cref{rmk:equiv_heap}).

\begin{lemma}
\label{lem:push_down}
Given a heap $H = \ev{ \Omega,\leq, \ell}$ and an element $\omega\in \Omega$, then the down-set $\downset[ ]{\omega}$ is a pyramid. Letting $H_1:=\downset[]{\omega}$ and $H_2=H\setminus H_1$, we have $\mathcal{M}\ev{ H_2 } = \mathcal{M} \ev{ H } \setminus \{\omega\}$ and
$$
H_1\circ H_2 = H
$$
\end{lemma}
\begin{proof}
Shown in the appendix of \cite{co2}.
\end{proof}

As a consequence of \Cref{lem:push_down}, we note the following: Given a heap $H$ with $\mathcal{M}\ev{H}=\{\omega_1,\ldots,\omega_k\}$, then we can define $P_1 := \downset[H]{\omega_1}$ and recursively obtain $k$ pyramids $\{P_i\}_{i=1}^{k}$, by defining $P_{i+1}:=\downset{\omega_{i+1}} $ in the remainder heap $H\setminus (P_1 \cup \ldots \cup P_i)$, for each $i=1,\ldots, k-1$. The composition of pyramids obtained in this way gives back the original heap:
$$H = P_1\circ \ldots \circ P_k$$
Here, depending on the order in which we ``push" the pieces $\{\omega_i\}_{i=1}^{k}$ down, we obtain at most $k!$ many decompositions of $H$ into $k$ pyramids (see \cite[Theorem 4.4]{krattenthaler}).

The definition of composition of two heaps is found in \cite[Definition 2.5, p.~4]{krattenthaler} (also see \cite{co2}). Here, we provide a simplified version for full heaps: 
\begin{definition}
\label{def:composition_of_heaps}
Let $H_1 = (\Omega_1, \leq_1 ) , H_2 = (\Omega_2, \leq_2 ) \in \mathbbm{H}\ev{ \mathcal{B}, \mathcal{R}}$ be full heaps, such that $\Omega_1 \cap \Omega_2 = \emptyset$. Then
the composition of $H_1$ and $H_2$, denoted $H_1 \circ H_2$, is the heap $(\Omega_1 \sqcup \Omega_2, \leq_3)$, where the partial order $\leq_3$ is the transitive closure of the relations:
\begin{enumerate}
\item[a)] $v_1 \leq_3 v_2$ if $v_1,v_2 \in \Omega_1$ and $v_1 \leq_1 v_2$,
\item[b)] $v_1 \leq_3 v_2$ if $v_1,v_2 \in \Omega_2$ and $v_1 \leq_2 v_2$,
\item[c)] $v_1 \leq_3 v_2$ if $v_1 \in \Omega_1$, $v_2 \in \Omega_2$ and $v_1 \mathcal{R} v_2$.
\end{enumerate} 
\end{definition}
The composition operation satisfies associativity, which makes $\mymathbb{h}\ev{ \mathcal{B}, \mathcal{R}}$ into a monoid, where the identity element is the empty heap. 

\begin{definition}[Connected set of pieces]
\label{def:connected_pieces}
Let $\mathcal{B}$ be a finite set (of pieces), with concurrence relation $\mathcal{R}$. We say that $\mathcal{B}$ is \textit{connected}, provided the graph $G_{\mathcal{B},\mathcal{R}}$ is connected. We define an indicator function on the subsets of $\mathcal{B}$, by $$
I\ev{A} = \begin{cases}
1 & \text{ if $A$ is connected} \\ 
0 & \text{ otherwise}
\end{cases}
$$
for any $A\subseteq \mathcal{B}$.
\end{definition}
Note that a pair $\ev{ \mathcal{B}, \mathcal{R}}$ can not have a full pyramid, unless it is connected. Next, we show that for a connected set of pieces $\mathcal{B}$, the cardinality of $\mathbb{P}^{\beta}\ev{ \mathcal{B}, \mathcal{R}}$ is independent of the choice of $\beta$. In other words, no matter which piece we choose, we can form the same number of pyramids with the chosen piece being maximal.

\begin{lemma}
\label{lem:mobius_pyramids}
Given a connected set of pieces $\mathcal{B} $ and two fixed concurrent pieces $\beta_1,\beta_2\in \mathcal{B}$, then, 
$$ | \mathbb{P}^{ \beta_2 }\ev{ \mathcal{B}  } | = \sum_{ \substack{ \{ \mathcal{B}_1, \mathcal{B}_2 \} : \mathcal{B}_1 \sqcup \mathcal{B}_2 = \mathcal{B}  \\  \beta_i \in  \mathcal{B}_i, \ I(\mathcal{B}_i) = 1  } }  | \mathbb{P}^{\beta_1}\ev{ \mathcal{B}_1  } | \cdot | \mathbb{P}^{\beta_2}\ev{ \mathcal{B}_2  } |  $$
\end{lemma}
\begin{proof}
We claim that the mapping:
\begin{align*}
\phi: \mathbb{P}^{ \beta_2 }\ev{ \mathcal{B}   } & \rightarrow \bigsqcup_{ \substack{ \{ \mathcal{B}_1, \mathcal{B}_2 \} : \mathcal{B}_1 \sqcup \mathcal{B}_2 = \mathcal{B}  \\  \beta_i \in  \mathcal{B}_i, \ I(\mathcal{B}_i) = 1  } } \mathbb{P}^{\beta_1}\ev{ \mathcal{B}_1  } \times \mathbb{P}^{\beta_2}\ev{ \mathcal{B}_2  } \\ 
P & \mapsto \ev{ \downset[P]{\beta_1}, P\setminus \downset[P]{\beta_1} }
\end{align*}
is a bijection, with inverse function $\psi: \ev{ P_1, P_2 } \mapsto P_1 \circ P_2$, where $\circ$ denotes composition of heaps. 
\begin{enumerate}
\item In one direction, we have
\begin{align*}
\phi\ev{ \psi\ev{ P_1, P_2 } } = \phi\ev{ P_1 \circ P_2 } = \ev{ \downset[P]{\beta_1}, P\setminus \downset[P]{\beta_1} }
\end{align*}
where $P := P_1\circ P_2$. First, we claim $P_1 = \downset[P]{\beta_1}$. 

($\subseteq$): Let an element $x\in P_1$ be chosen. By hypothesis, we have $\mathcal{M}\ev{P_1} = \{\beta_1\}$. So, $x\leq_{ P_1 } \beta_1$. By the definition of the composition $\circ $ of heaps, we obtain 
$ x \leq_{P_1\circ P_2} \beta_1$, in other words, $x\in \downset[P_1\circ P_2]{\beta_1}$. 

($\supseteq$): Let an element $x\in \downset[P_1\circ P_2]{\beta_1}$ be chosen. In particular, we have $x\leq_{P} \beta_1$. Then, we can find a sequence $\{y_i\}_{i=0}^{k}$ such that $x = y_k \lessdot_{P} \ldots \lessdot_{P} y_0 = \beta_1$. By induction on $i$, we will show that $y_i \in P_1$, for each $i=0,\ldots,k$. The basis step, $i=0$ is clear, since $\beta_1\in P_1$. For the inductive step, let $i\geq 1$ be fixed. Note that $y_i\in P_1$ and $y_{i+1} \lessdot_{P_1 \circ P_2} y_i$ implies $y_{i+1}\in P_1$, by the definition of composition of heaps. 

It follows, easily, that $P\setminus \downset[P]{\beta_1} = P \setminus P_1 = (P_1\sqcup P_2)\setminus P_1 = P_2$.

\item The other direction $\psi\ev{ \phi\ev{ P } } = \phi\ev{  \downset[P]{\beta_1} \circ P\setminus \downset[P]{\beta_1} } = P$, follows from \Cref{lem:push_down}.
\end{enumerate}
\end{proof}

\begin{corollary}
\label{lem:full_pyramids_balanced}
Given a connected set of pieces $\mathcal{B}$, then the number $|\mathbb{P}^{\beta}\ev{\mathcal{B}}|$ of full pyramids with maximal piece $\beta$ is constant over the variable $\beta\in\mathcal{B}$. 
\end{corollary}
\begin{proof}
Since $\mathcal{B}$ is connected, it is enough to show the statement for two concurrent pieces $\beta_1$ and $\beta_2$. Assume that $\beta_1$ and $\beta_2$ are concurrent. Since the right-hand side of the equation in \Cref{lem:mobius_pyramids} is symmetric, the statement follows:
\begin{align*}
& | \mathbb{P}^{ \beta_2 }\ev{ \mathcal{B}  } | = \sum_{ \substack{ \{ \mathcal{B}_1, \mathcal{B}_2 \} : \mathcal{B}_1 \sqcup \mathcal{B}_2 = \mathcal{B}  \\  \beta_i \in  \mathcal{B}_i, \ I(\mathcal{B}_i) = 1  } }  | \mathbb{P}^{\beta_1}\ev{ \mathcal{B}_1  } | \cdot | \mathbb{P}^{\beta_2}\ev{ \mathcal{B}_2  } |  = \sum_{ \substack{ \{ \mathcal{B}_2, \mathcal{B}_1 \} : \mathcal{B}_2 \sqcup \mathcal{B}_1 = \mathcal{B}  \\  \beta_i \in  \mathcal{B}_i, \ I(\mathcal{B}_i) = 1  } }  | \mathbb{P}^{\beta_2}\ev{ \mathcal{B}_2 } | \cdot | \mathbb{P}^{\beta_1}\ev{ \mathcal{B}_1  } |  = | \mathbb{P}^{ \beta_1 }\ev{ \mathcal{B}  } | 
\end{align*}
\end{proof}
In particular, given a fixed piece $\beta\in \mathcal{B}$, we can calculate the total number of full pyramids, via the following equation:
$$ | \mathbb{P} \ev{ \mathcal{B} } |  = \sum_{ \gamma\in \mathcal{B}}  | \mathbb{P}^{ \gamma }\ev{ \mathcal{B} } | =  | \mathcal{B}  | \cdot | \mathbb{P}^{ \beta }\ev{ \mathcal{B} } | $$

\subsection{The Bond Lattice}
Let $L$ be a finite lattice. We write $A\ev{L}$ for the set of atoms of $L$ (see \Cref{def:posets_prelim}). Given a set of atoms $D\subseteq A\ev{L}$, we say that $D$ is a \textit{base} of $x:=\bigvee_{y\in D} y$. We shorten the notation and write $\bigvee D$ instead of $\bigvee_{y\in D} y$.

We will provide necessary definitions, to state the NBC Theorem of Rota (q.v.~\Cref{thm:rota_NBC_thm} below). Let $L$ be a finite geometric lattice, with rank function $\textup{rk}$. It is well-known (\cite[p.~86]{birkhoff}) that for any $B\subseteq A\ev{L}$, we have $\textup{rk}\ev{ \bigvee B } \leq |B|$. 
\begin{definition}
Let $\trianglelefteq$ be any fixed linear order on $A\ev{L}$. 
\label{def:NBC_sets}
\begin{enumerate}
\item Given $B\subseteq A\ev{L}$, then $B$ is \textit{independent}, provided $\textup{rk}\ev{\bigvee B} = |B|$. 
\item A minimal (with respect to inclusion) dependent set $C\subseteq A\ev{L}$ is called a \textit{circuit}. 
\item Given a circuit $C \subseteq A\ev{L}$, let $c$ be the maximal element of $C$ with respect to $\trianglelefteq$. Then, $C-c$ is called a \textit{broken-circuit}. 
\item An set $B\subseteq A\ev{L}$ is \textit{NBC}, if $B$ has no subset that is a broken-circuit. Let $\mathcal{N}^{*}\ev{L}$ be the set of NBC subsets of $A\ev{L}$. If $T \in \mathcal{N}^{*}\ev{L} $ is an NBC set, then $T$ is an \textit{NBC base} for $x = \bigvee T $. Let $\mathcal{N}^{x}\ev{L} = \{ T \in \mathcal{N}^{*}\ev{ L }: \bigvee T = x \}$ be the set of NBC bases of an element $x\in L$. We shorten the notation and write $\mathcal{N}\ev{L}$, instead of $\mathcal{N}^{\mymathbb{1}_L}\ev{L}$, for the set of NBC bases of $\mymathbb{1}_L$. 
\end{enumerate}	
\end{definition}
\begin{remark}

\phantom{a}

\begin{enumerate}
\item Broken-circuits are defined differently in \cite{sagan}, where $C-c$ is called a broken-circuit, if $c$ is minimal. Instead, we follow the notation of \cite{benson} in \Cref{def:NBC_sets}. (The dual theory is obtained easily, by switching the orders in what follows, and by using orientations with a unique source, instead of with a unique sink.) 
\item An NBC set is necessarily independent. 
\end{enumerate}
\end{remark}
We are ready to state the NBC Theorem of Rota (\cite[Proposition 1, p.~358]{rota}):
\begin{theorem}
\label{thm:rota_NBC_thm}
Let $L$ be a finite geometric lattice, with a rank function $\textup{rk}$. Let $\trianglelefteq$ be a linear order on $A\ev{L}$. Then, the M\"{o}bius function satisfies:
$$
\mu\ev{ \mymathbb{0}, x } =  \ev{-1}^{\textup{rk}\ev{x}} \cdot \left| \mathcal{N}^{x}\ev{L} \right| 
$$
for any $x\in L$. 
\end{theorem}

The definition of the bond lattice of a graph $G$ can be found in \cite{tsuchiya}, restated below: 
\begin{definition}
\label{def:bond_lattice_graph}
Let $G$ be a simple graph. 
\begin{enumerate}
\item The \textit{closure} of subset $F \subseteq E\ev{G}$ is $\overline{F} := \{ e\in E\ev{G}: V\ev{e} \subseteq F\}$. A subset $F$ is \textit{closed}, if $\overline{F} = F$. Let $L\ev{G} := \{ F\subseteq E\ev{G} : F \text{ is closed}\}$.  
\item The \textit{bond lattice} of $G$ is $\ev{ L\ev{G}, \wedge, \vee}$, where $F_1 \vee F_2 := \overline{ F_1 \cup F_2 }$ and $F_1 \wedge F_2 := F_1 \cap F_2$. 
\end{enumerate}
\end{definition}

Note that the atoms of the bond lattice $L\ev{G}$ is the set of edges $E\ev{G}$. Given a linear order $\trianglelefteq$ on $E\ev{G}$, it follows by the definition of a circuit, that circuits of $L\ev{G}$ are exactly cycles of length $\geq 3$ of $G$, and a subset $D \subseteq E\ev{G}$ is a base of $\mymathbb{1}_{L\ev{G}}$ if and only if $D$ is a spanning tree. So, NBC bases of $\mymathbb{1}_{L\ev{G}}$ are exactly the spanning trees of $G$ that do not contain any broken-circuits. From now on, we will shortly write $\mathcal{N}\ev{G}$ instead of $\mathcal{N}\ev{L\ev{G}}$, for the set of NBC bases of $\mymathbb{1}_{L\ev{G}}$.

For a pair $\ev{\mathcal{B},\mathcal{R}}$ of pieces and a concurrence relation, consider the partition lattice $\Pi\ev{ \mathcal{B} }$. We will restrict to the sublattice induced by partitions into connected parts, to obtain the bond lattice of $G_{\mathcal{B},\mathcal{R}}$. More formally, we define:
\begin{definition}
By an abuse of notation, let us extend the indicator function $I$ of \Cref{def:connected_pieces} to the partition lattice $\Pi\ev{ \mathcal{B} }$, by 
$$
I\ev{ \{ A_1,\ldots, A_k \}} = \prod_{i=1}^{k} I\ev{ A_i }
$$
\end{definition}
Note that for an element $a = \{A_1,\ldots,A_k\} \in \Pi\ev{ \mathcal{B} }$, we have $I\ev{ a } \neq 0$ if and only if every part of $a$ is connected. Let $\textup{Conn}\ev{\mathcal{B}}$ be the sublattice of $\Pi\ev{ \mathcal{B} }$ induced by partitions into connected parts:
$$
\textup{Conn}\ev{\mathcal{B}} = \{ a = \{ A_1,\ldots,A_k\}  \in \Pi\ev{ \mathcal{B}} : I \ev{a} \neq 0 \}
$$
Then, the mapping $\textup{Conn}\ev{\mathcal{B}} \rightarrow L\ev{G} $ given by $a = \{ A_1,\ldots,A_k\} \mapsto \bigcup_{i=1}^{k} E\ev{ G[A_i] }$ is an isomorphism of lattices, where $G=G_{\mathcal{B},\mathcal{R}}$ and $G[A]$ is the subgraph of $G$ induced by $A$ for each $A\subseteq V\ev{G}$. Therefore, the bond lattice of $G_{\mathcal{B},\mathcal{R}}$ will be identified with the sublattice of the partition lattice induced by partitions into connected parts. The \textit{bond lattice of a set of pieces} $\mathcal{B}$, with concurrence relation $\mathcal{R}$ is defined as the bond lattice $L\ev{ G_{\mathcal{B},\mathcal{R}} }$, denoted shortly $L\ev{ \mathcal{B} }$. 

\subsection{Bijections between NBC Bases of \texorpdfstring{$L(G)$}{1L(G)} and Acyclic Orientations of $G$ with a Unique Sink}
Consider a simple graph $G$. When we combine the results in \cite{zaslavsky}, which connects the chromatic polynomial and the acyclic unique sink orientations, (also see \cite{stanley}) and Whitney's result about the calculation of the chromatic polynomial via NBC sets, we obtain that acyclic orientations of $G$ with a unique sink equal in cardinality to the NBC bases of $\mymathbb{1}_{L\ev{G}}$, in the bond lattice $L\ev{G}$. Furthermore, direct mappings between the two sets are given in \cite{benson, sagan_gebhard, biggs_winkler}. Here, we summarize the construction given in \cite{benson}, where the forward function $\mu$ is explicit and the inverse function $\pi$ is recursive. Afterwards, we give alternative recursive definitions of $\mu$ and $\pi$. 
\begin{theorem}
\label{thm:NBC_bij_explicit}
Let $G$ be a graph. For any vertex $x\in V\ev{G}$, there exist inverse functions:
\begin{center}
\begin{tikzcd}
\mathcal{N}\ev{ G } \arrow[r,shift left=2pt,"\mu"] & \arrow[l,shift left=2pt,"\pi"] \overline{\mathcal{O}}^{x}\ev{ G } 
\end{tikzcd}
\end{center}

\end{theorem}
\begin{proof}
See \cite[Theorem 4.2, p.~11]{benson}.
\end{proof}

We summarize the construction of $\mu$: Let $G$ be a graph, with a linear order $\trianglelefteq$ on $E\ev{G}$. Given an NBC base $T \in \mathcal{N}\ev{G}$, consider the arborescence where $x$ is the root. 

Given two vertices $i,j \in V\ev{T}$:
\begin{enumerate}
\item The unique path from $i$ to the root $x$ is will be denoted $R_i$. 
\item The \textit{meet} of $i$ and $j$, denoted $\text{meet}\ev{i,j}$ is the unique vertex in $V\ev{R_i}\cap V\ev{R_j}$, that is farthest from $x$. (If $j\in V\ev{R_i}$, then $\text{meet}\ev{i,j}=j$.)
\item Let $e_{ij}$ be the maximal edge of $E\ev{R_i} \setminus E\ev{R_{\text{meet}\ev{i,j}}}$ with respect to $\trianglelefteq$. (Whenever $j\in V\ev{R_i}$, we define $e_{ji}$ to be the null-edge, which is less than every other edge.) 
\item Let $j \prec i$ if and only if $e_{ji} \triangleleft e_{ij}$, i.e., the largest edge from $i$ to the meet is larger than the largest edge from $j$ to the meet. (If $j\in P_i$, then $e_{ji} = \text{null-edge} \triangleleft e_{ij}$, so it automatically follows that $j \prec i$. In particular, $x$ is less than every other vertex, with respect to $\prec$.) Note that $\prec$ is a linear order on the set of vertices. 
\item Define the unique sink orientation $\mu\ev{T}$, by including the edges $\ev{i,j}$ for each pair $i,j$ such that $j \prec i$. In other words, the arrows point to the vertices with smaller assigned edge-value. 
\end{enumerate}

Next, we aim to define a recursive function $\phi'$, and show that $\phi' = \mu$. First, we have a note on the recursive structure of the NBC bases of $\mymathbb{1}_{L\ev{G}}$: 
\begin{lemma}
\label{lem:NBC_decompose}
Let $G$ be a simple graph. Consider a linear order $\trianglelefteq$ on edge-set $E\ev{G}$. Let $e = \{v_1,v_2\}$ be the maximal element of $E\ev{ G }$ with respect to $\trianglelefteq$. Then, the following map is well-defined and bijective:
\begin{align*}
\phi: \mathcal{N}\ev{ G } & \rightarrow \bigsqcup_{\substack{ \{ S_1, S_2 \},\ S_1\sqcup S_2 = V\ev{G} \\ v_i\in S_i,\ G[S_i] \textup{ connected} } }  \mathcal{N}\ev{ G[S_1] } \times \mathcal{N}\ev{ G[S_2] }  \\ 
T & \mapsto \ev{ T_1, T_2 } 
\end{align*}
where $T - e = T_1 \sqcup T_2$ and $S_i = V\ev{T_i}$, for $i=1,2$.
\end{lemma}
\begin{proof}
To see that $\phi$ is well-defined, note that any spanning tree of $G$ that does not contain a broken-circuit must contain $e$. So, $T-e$ is a disjoint union of trees, $T_1$ and $T_2$. As $T$ does not have any broken-circuit, it follows that $T_i$ is an NBC set of $G[S_i]$, for $i=1,2$. 

For surjectivity, let $\ev{T_1,T_2}$ be an element of the codomain of $\phi$. Let $T := e\cup T_1 \cup T_2$. Suppose, for a contradiction, that $T$ contains a broken-circuit $C' = C-f$, where $f\in E\ev{G}$ and $C$ is a circuit. Since $T_i$ are both NBC sets, it follows that $e\in E\ev{C}$. By definition of a broken-circuit, we get $e\triangleleft f$, which contradicts the definition of $e$. 

For injectivity, let $\phi\ev{T} = \phi\ev{T'} = \ev{ T_1 ,T_2 }$. In particular, $T-e = T_1 \cup T_2 = T'-e$. This implies that $T = e \cup T_1 \cup T_2 = T'$. 
\end{proof}

\Cref{lem:NBC_decompose} can be turned into an algorithm that recursively calculates the NBC bases of $\mymathbb{1}_{L\ev{G}}$ for a graph $G$. Similar bijections $\mathcal{N}\ev{ G } \leftrightarrow \mathcal{N}\ev{ G \big\vert_{e} }$ and $\mathcal{N}\ev{ G } \leftrightarrow \mathcal{N}\ev{ G - e }$ (where $G \big\vert_{e}$ and $G - e$ denote contraction and deletion, respectively) are found in \cite{dohmen}, where an inductive proof of Whitney's Theorem is given. 

\begin{lemma}
\label{lem:bij_NBC_pyramid}
Let $G$ be a simple graph. Consider a linear order $\trianglelefteq$ on the edge-set $E\ev{G}$. Fix any vertex $x \in V\ev{G}$. Then, we have a bijection between NBC bases of $\mymathbb{1}_{L\ev{G}}$ and the full pyramids with maximal piece $x$: 
$$ 
\mathcal{N}\ev{ G } \rightarrow \mymathbb{P}^{x}\ev{ \mathcal{B}_G,\mathcal{R}_G } 
$$
\end{lemma}
\begin{proof}
For each $i\geq 1$, for each connected subset $S\in \binom{V\ev{G}}{i}$ and for each $x\in S$, we will recursively define maps $\phi$ and $\psi$ which are inverses of each other:
\[
\begin{tikzcd}
\mathcal{N}\ev{ G[S] } \arrow[r,shift left=2pt,"\phi_{S,x}"] & \arrow[l,shift left=2pt,"\psi_{S,x}"] \mymathbb{P}^{x}\ev{ S } 
\end{tikzcd}
\]
For each $S\subseteq V\ev{G}$, let $e_S$ be the maximal element of $E\ev{G[S]}$ with respect to $\trianglelefteq$. For $S\in \binom{V\ev{G}}{1}$, let $\phi_{S,x}\ev{x} = \psi_{S,x}\ev{x} = x$ be the singleton. 

Consider any connected $S\in \binom{V\ev{G}}{i}$, for $i\geq 2$ and $x\in S$. Define:
\begin{align*}
\phi_{S,x}: \mathcal{N}\ev{ G[S] } &\rightarrow \mymathbb{P}^{x}\ev{ S } \\ 
T & \mapsto \phi_{S_1,u}\ev{ T_1 } \circ \phi_{S_2,x} \ev{ T_2 }\\
& \text{where $T-e_S=T_1 \sqcup T_2$, $S_i = V\ev{T_i}$ and $x\in V\ev{T_2}$ and $S_1 \cap e_S = \{u\}$}
\end{align*}

Conversely, for each pyramid $P$, the endpoints of $e_S$ are comparable with respect to $\leq_P$, by the definition of a heap, so let $e_S = \{y_S^{P},z_S^{P}\}$ such that $y_S^{P} \leq_P z_S^{P}$. Define:
\begin{align*}
\psi_{S,x}:  \mymathbb{P}^{x}\ev{ S } & \rightarrow \mathcal{N}\ev{ G[S] } \\ 
P & \mapsto e_S \cup \psi_{S_1, y_S^{P}}\ev{ P_1 } \cup \psi_{S_2, x}\ev{ P_2 }\\
&\text{where $P_1 = \downset[P]{y_S^{P}}$, $ P_2 = P\setminus P_1$ and $S_i = V\ev{ P_i }$.}
\end{align*}
Now, we show that $\psi_{S,x} $ and $\phi_{S,x}$ are well-defined and inverses of each other, by induction on $|S|$. The base case $|S| = 1 $  is clear. By the inductive hypothesis, we know that the following are inverses:
$$
\begin{tikzcd} \mathcal{N}\ev{ G[S] }  \arrow[r,shift left=2pt,"\phi_{S,y}"] & \arrow[l,shift left=2pt,"\psi_{S,y}"]  \mymathbb{P}^{y}\ev{ S } \end{tikzcd} 
$$
for any $S\in \binom{V\ev{G}}{j}$, $j\leq i-1$ and $y\in S$. For the inductive step, we first claim $\psi_{S,x} \phi_{S,x} = \textbf{id}$. Let $T$ be an NBC base of $\mymathbb{1}_{G[S]}$. Then, $
\phi_{S,x}\ev{ T } = P_1  \circ P_2  =: P $, where: 
\begin{enumerate}
\item $T-e_S = T_1 \cup T_2$ and $S_i = V\ev{ T_i }$
\item $P_1 = \phi_{S_1,u} \ev{ T_1 }$ and $P_2 = \phi_{S_2,x}\ev{ T_2 }$
\item $x\in S_2$, $e_S = \{ u,v\}$ and $S_1 \cap e = \{u\}$. In particular, we have $v\in S_2$.
\end{enumerate}

By \Cref{lem:NBC_decompose}, we have $T_i \in \mathcal{N}\ev{ G[S_i] } $, which implies, by the inductive hypothesis, that $ P_1 \in \mymathbb{P}^{u}\ev{ S_1 }$, $P_2 \in \mymathbb{P}^{x}\ev{ S_2 } $ and $u \leq_P v \leq_P x$. Hence, $\phi_S\ev{T}$ is well-defined. 

By the definition of $\psi_S$, it follows that 
$$
\psi_{S,x} \phi_{S,x}\ev{ T } = \psi_{S,x}\ev{P} = e_S \cup \psi_{S_1,u}\ev{P_1} \cup \psi_{S_2,x}\ev{ P_2 } = e_S \cup T_1 \cup T_2 = T
$$ 

Next, we claim that $\phi_{S,x} \psi_{S,x} = \textbf{id}$. Let $P$ be a pyramid with maximal piece $x$. Then, $\psi_{S,x}\ev{ P } = e_S \cup T_1 \cup T_2 =: T$, where: 
\begin{enumerate}
\item $e_S = \{y_S^{P},z_S^{P}\}$ with $y_S^{P}\leq_P z_S^{P}$
\item $P_1 = \downset[P]{y_S^{P}}$ and $P_2 = P\setminus P_1$ and $S_i = V\ev{ P_i }$. 
\item $T_1 = \psi_{S_1}\ev{ P_1 }$ and $T_2 = \psi_{S_2}\ev{ P_2 }$. 
\end{enumerate}
By the inductive hypothesis, $\psi_{S_i}\ev{ P_i }$ is an NBC base of $G[S_i]$, for $i=1,2$. Suppose, for a contradiction, that $T$ contains a broken-circuit $C' = C-f$, for a cycle $C$. Then, $e_S\in C'$, and so $e_S\triangleleft f$, contradicting the definition of $e_S$. Hence, $\psi_S\ev{P}$ is well-defined. 

Note that by anti-symmetry condition, $y_S^{P}\leq_P z_S^{P}$ and $y_S^{P} \neq z_S^{P}$ imply $z_S^{P} \not \leq_P y_S^{P}$, and so, $z_S^{P} \in S_2$, which means that $e_S \cap S_1 = \{y_S\}$, so, by the definition of $\phi_S$, we have 
$$
\phi_{S,x} \psi_{S,x}\ev{ P } = \phi_{S,x}\ev{ T } = \phi_{S_1, y_S}\ev{ T_1 } \circ \phi_{S_2,x}\ev{ T_2 } = P_1 \circ P_2 = P
$$
\end{proof}

As noted in \cite[p.~325]{viennot}, there is a natural bijection between pyramids with maximal piece $x$ and acyclic orientations of $G$ with unique sink $x$. For a given full pyramid $P\in \mymathbb{P}^{x}\ev{\mathcal{B}_{G},\mathcal{R}_{G}}$, let $O_P$ be the acyclic orientation of $G$ obtained by including the edge $\ev{y,z}$ for each edge $\{ y,z\}\in E\ev{G}$ with $y\leq_{P} z$. (We draw an arrow from the smaller piece to the larger, with respect to $\leq_P$.) Note that $O_P$ is well-defined, as concurrent pieces are comparable in a heap. Conversely, given an orientation with unique sink $x$, we obtain a pyramid $P_O$ with maximal piece $x$, by taking the transitive closure of the covering relations $ \{ y\lessdot_P z : \ev{ y, z} \in E\ev{G}\}$. (Tails of the arrows are designated to be larger than heads.) We note that the pyramid $P_O$ is well-defined, as $O$ is acyclic. Consider the mappings 
$$
\begin{tikzcd}
\mathcal{N}\ev{ G } \arrow[r,shift left=2pt,"\phi'"] & \arrow[l,shift left=2pt,"\psi' "] \mymathbb{P}^{x}\ev{ G } 
\end{tikzcd}
$$
where $\phi'\ev{T} = O_{\phi\ev{T}}$ and $\psi'\ev{ O } = \psi\ev{ P_O}$.

To prove:
\begin{equation}
\label{eqn:equal_bij}
\mu = \phi'
\end{equation}
we apply induction on $|S|$ and show:
\begin{equation}
\label{eqn:inductive_claim}
\mu\ev{T} [S] = \phi_{S} \ev{ T[S] }
\end{equation}
for each set of vertices $S\subseteq V\ev{G}$. The base case is clear. For the inductive step, let $e_S$ be the maximal edge in $E\ev{T[S]}$, with respect to $\triangleleft$ and $T[S] = T_1 \sqcup T_2 \sqcup e_S$, where $T_i = T[S_i]$, for $i=1,2$. By the inductive hypothesis, we have:
\begin{equation}
\label{eqn:inductive_hyp}
O_i := \mu\ev{T}[S_i] = \phi'_{S_i}\ev{ T_i }
\end{equation}
for $i=1,2$. Let $P:= \phi_{S}\ev{ T[S] }$ be the pyramid output by $\phi$. It suffices to show that the edges $E\ev{S_1,S_2}$ are oriented in the same direction, in the unique sink oriented digraphs $\mu\ev{T}[S]$ and $\phi'_{S} \ev{ T }$. Given any two vertices $x_i \in S_i$ (for $i=1,2$), note that the unique sink $x$ belongs to $S_2$, and so, the maximal edge $e_S$ is on the path $R_1 \subseteq T[S]$, which means $e_{x_2 x_1 } \triangleleft e_{x_1 x_2} = e_S$, in other words, $x_2 \prec x_1$, i.e., $\ev{x_2, x_1} \in E\ev{\mu\ev{ T } [S] }$. On the other hand, we have $x_1 \leq_{P} x_2$, and so, $\ev{x_2,x_1} \in E\ev{ O_P } = \phi'_{S} \ev{ T }$, and the proof is complete. 

We finish this section with two lemmas, to be used in the next section on circuits of digraphs. 
\begin{lemma}
\label{lem:heap_lattice_product_poset_isom}
Given a set of pieces $\mathcal{B} $, consider the partition lattice $\Pi\ev{\mathcal{B}}$ and the sublattice $L\ev{\mathcal{B}}$ of $\Pi\ev{\mathcal{B}}$. Given an element $b = \{B_1,\ldots,B_k\} \in \Pi\ev{ \mathcal{B} }$, then the following map is a lattice isomorphism:
\begin{align*}
\phi :   \Pi\ev{B_1} \times \cdots \times \Pi\ev{B_k} &\rightarrow \downset[ \Pi\ev{ \mathcal{B} } ]{b} \\ 
\ev{ c_1,\ldots,c_k } & \mapsto \bigcup_{i=1}^{t} c_k
\end{align*}
Furthermore, for each element $b = \{B_1,\ldots,B_k\} \in L\ev{ \mathcal{B} }$, the map below is a lattice isomorphism:
\begin{align*}
\phi' : L\ev{B_1} \times \cdots \times L\ev{B_k} &\rightarrow \downset[L\ev{ \mathcal{B} } ]{b} \\ 
\ev{ c_1,\ldots,c_k } & \mapsto \bigcup_{i=1}^{k} c_i
\end{align*}
\end{lemma}
\begin{proof}
It is clear that $\phi$ is a lattice isomorphism (C.f. \cite[p.~344]{vanlint} for a short argument.). By the definition of the bond lattice, the restriction $\phi'$ respects the partial order relations, $\leq_{L\ev{B_1} \times \cdots \times L\ev{B_k}}$ and $\leq_{ \downset[L\ev{\mathcal{B}}]{b} }$. 
\end{proof}

The following lemma can be derived by combining \Cref{thm:rota_NBC_thm} and \Cref{thm:NBC_bij_explicit}. For reference, we provide an alternative proof, using Weisner's theorem, in the language of Heaps of Pieces.
\begin{lemma}
\label{lem:balanced_pieces_mobius}
For a set of pieces $\mathcal{B} = \{\beta_1,\ldots,\beta_k\}$, the M\"{o}bius function of the bond lattice of $\mathcal{B}$ satisfies: 
$$
\mu_{L\ev{ \mathcal{B} }}\ev{ \mymathbb{0}, \mymathbb{1} } = \ev{ -1 }^{1-k} |\mymathbb{P}^{\beta}\ev{\mathcal{B},\mathcal{R}}|
$$
for each $\beta \in \mathcal{B}$.
\end{lemma}
\begin{proof}
Fix two concurrent pieces $\beta_1,\beta_2$ and consider the point $p:= \{ \beta_1 \beta_2, \beta_3, \ldots, \beta_t \}$ of $L\ev{ \mathcal{B} }$. By Weisner’s Theorem (\cite[Theorem 25.3 on p.~339, Equation (25.6) on p.~340]{vanlint}), we have 
$$ \mu\ev{ \mymathbb{0}, \mathbbm{1} } = - \sum_{ \substack{h: h \lessdot \mathbbm{1} \\ p\not \leq h } } \mu\ev{ \mymathbb{0}, h }$$
\begin{enumerate}
\item The copoints $h$ of $L\ev{\mathcal{B}}$ are in the form $h = \{ A_1, A_2\}$, where $A_i$ is connected (for $i=1,2$) and $\mathcal{B} = A_1 \sqcup A_2$. 
\item Given a copoint $ h $, then the point $p$ does not refine $h$ if and only if the pieces $\beta_1$ and $\beta_2 $ belong to different parts of $h$.
\item Given a copoint $h = \{A_1,A_2\}$, then by \Cref{lem:heap_lattice_product_poset_isom}, we have a poset isomorphism $$L\ev{A_1} \times L\ev{ A_2 } \rightarrow \downset{h}$$ By \cite[p.~344]{vanlint}, we have the property 
$$
\mu_{L\ev{A_1}}\ev{ a_1, \mymathbb{1}_{A_1} } \cdot \mu_{ L\ev{A_2} } \ev{ a_2, \mymathbb{1}_{A_2} } = \mu_{ \downset{h} }\ev{ a_1 \cup a_2, h }
$$
for each $\ev{ a_1,a_2 } \in L\ev{A_1} \times L\ev{A_2}$. 
\end{enumerate}

We are ready to apply induction on the number of pieces $|\mathcal{B}| = k $, to show the statement, $\mu\ev{ \mymathbb{0}, \mathbbm{1} } = \ev{ -1 }^{ 1 - k } | \mathbb{P}^{ \beta }\ev{ \mathcal{B} ,\mathcal{R}} | $, for any $\beta \in \mathcal{B}$. For the basis step, let $k = 1$. Then, $L\ev{ \mathcal{B} } = \{ \beta_1 \}$ is a singleton, and so, $ \mymathbb{0} = \mathbbm{1} $ and $ \mu\ev{ \mymathbb{0}, \mathbbm{1} } = 1 = \ev{ -1 }^{ 1 - 1 }| \mathbb{P}^{ \beta_1 }\ev{ \mathcal{B} ,\mathcal{R}} | $. For the inductive step, we calculate,
\begin{align*}
& \mu\ev{ \mymathbb{0}, \mathbbm{1}  } = - \sum_{ \substack{h: h \lessdot \mathbbm{1} \\ p\not \leq h } } \mu_{L\ev{ \mathcal{B} }}\ev{ \mymathbb{0}, h } \\ 
& = - \sum_{ \substack{h: h \lessdot \mathbbm{1} \\ p\not \leq h } } \mu_{ \downset{h} }\ev{ \mymathbb{0}, h } && \text{ by \Cref{lem:mu_down_set}}\\ 
& = -  \sum_{ \substack{ \{ A_1, A_2 \} : A_1 \sqcup A_2 = \mathcal{B}  \\  \beta_i \in  A_i, \ I(A_i) = 1  } }  \mu_{ L\ev{A_1} }\ev{ \mymathbb{0}_{ A_1 }, \mymathbb{1}_{ A_1 } }  \cdot \mu_{ L\ev{A_2} }\ev{ \mymathbb{0}_{ A_2 }, \mymathbb{1}_{ A_2 } }\\ 
& = -  \sum_{ \substack{ \{ A_1, A_2 \} : A_1 \sqcup A_2 = \mathcal{B}  \\  \beta_i \in  A_i, \ I(A_i) = 1  } }  \ev{ -1 }^{ 1 - | \mymathbb{0}_{ A_1 } | } | \mathbb{P}^{ \beta_1 }\ev{ A_1 ,\mathcal{R}} | \cdot \ev{ -1 }^{ 1 - | \mymathbb{0}_{ A_2 } | } | \mathbb{P}^{ \beta_2 }\ev{ A_2 ,\mathcal{R}} |  \\ 
&\hspace{2cm} \text{ by the inductive hypothesis } \\ 
& = \ev{ -1 }^{ 1 - k } \sum_{ \substack{ \{ A_1, A_2 \} : \beta_i \in  A_i \\ I(A_i) = 1, \  A_1 \sqcup A_2 = \mathcal{B} } }  | \mathbb{P}^{ \beta_1 }\ev{ A_1 ,\mathcal{R}} |  \cdot  | \mathbb{P}^{ \beta_2 }\ev{ A_2 ,\mathcal{R}} | &&\text{ since $| \mymathbb{0}_{A_1} | + |\mymathbb{0}_{A_2} | = k$ }  \\
& = \ev{ -1 }^{ 1 - k } | \mathbb{P}^{ \beta_2 }\ev{ \mathcal{B} ,\mathcal{R}} | && \text{ by \Cref{lem:mobius_pyramids}}
\end{align*}
Hence, we have $\mu\ev{ \mymathbb{0}, \mathbbm{1}  }  = \ev{ -1 }^{ 1 - k } | \mathbb{P}^{ \beta_2 }\ev{ \mathcal{B} ,\mathcal{R}} |$. The statement follows, by \Cref{lem:full_pyramids_balanced}. 
\end{proof}

\section{Partition Lattice of the Edge-set of an Eulerian Digraph} 
\label{sec:partition_lattice_eulerian_digraph}
\subsection{Preliminaries on Digraphs and Closed Trails}
We outline some definitions for Eulerian digraphs. For more detail and examples, we refer the reader to \cite{co2}. Let $X = \ev{ V\ev{X}, E\ev{X}, \varphi }$ be a multi-graph of rank $k=2$. A \textit{walk} of $X$ is an alternating sequence \[\mathbf{w}=\ev{v_0,e_1,v_1,e_2,\ldots, v_{d-1}, e_{d}, v_d}\] of vertices $\ev{ v_i }_{i=0}^{d}$ and edges $\ev{ e_i }_{i=1}^{d}$, such that consecutive pairs of elements are incident. 

The \textit{length} of a walk $\mathbf{w}$, denoted $|\mathbf{w}|$, is the number of edges in $\mathbf{w}$. The vertices $v_0,v_d$ are said to be the \textit{initial vertex} and the \textit{terminal vertex} of $\mathbf{w}$, respectively. A \textit{closed walk} at $v_0$ is a walk where $v_0 = v_d$.  

A walk $\mathbf{w}$ is called a \textit{trail}, if its edges are distinct. A closed trail $\mathbf{w}$ of $X$ is an \textit{Eulerian trail}, if it uses all the edges of $X$. We denote by $\mathcal{T}\ev{X}$ and $\mathcal{W}\ev{X}$, the sets of closed trails and Eulerian (closed) trails of $X$, respectively. Given an edge $e\in E\ev{X}$ with $\varphi\ev{e} = \{ z_1,z_2 \}$, the set of \textit{Eulerian (closed) trails ending at $e$} is defined as,
\begin{align*}
\mathcal{W}^{e}\ev{X} &= \{ \mathbf{w} = \ev{ v_0, e_1, v_1,\ldots, v_{ d-2 }, e_{ d-1 }, z_1, e, z_2 } : \mathbf{w} \in \mathcal{W}\ev{X} \} \\
&\sqcup \{ \mathbf{w} = \ev{ v_0, e_1, v_1,\ldots, v_{ d-2 }, e_{ d-1 }, z_2, e, z_1 } : \mathbf{w} \in \mathcal{W}\ev{X} \}
\end{align*}

A walk of a digraph is defined similarly: A walk $\mathbf{w}=\ev{v_0,e_1,v_1,e_2,\ldots, v_{d-1}, e_{d}, v_d}$ must satisfy $\psi\ev{ e_i } = \ev{ v_{i-1}, v_{i} }$, for each $i = 1,\ldots, d$. 

Let $X$ be a multi-graph of rank $k=2$. We define the set of \textit{circuits} of $X$ as the quotient $\mathcal{Z}\ev{X} := \mathcal{T}\ev{X} /\textup{rk}$ of closed trails under the equivalence relation $\rho$, which is defined by cyclic permutation of the edges. More formally, $\rho$ is the transitive closure of the relations,
$$ \mathbf{w}_1 = \ev{ v_0, e_1, v_1, e_2,\ldots, e_{d-1}, v_{d-1}, e_{d}, v_0 } \ \ \rho \ \  \mathbf{w}_2 = \ev{ v_{d-1}, e_{d}, v_0, e_1, v_1,e_2,\ldots, e_{d-1} ,v_{d-1} } $$

The set of \textit{Eulerian circuits} is defined as $\mathfrak{C}\ev{ X } := \mathcal{W}\ev{ X } / \rho$, the quotient of the set of Eulerian trails of $X$, up to the equivalence relation of cyclic permutation of edges. 

Whenever a circuit $[w]_\rho \in \mathcal{Z}\ev{ X } $ has no repeated vertices, we call it a \textit{cycle}. Let $\mathcal{B}\ev{ X }$ denote the set of cycles of a multi-graph $X$. 

The circuits and cycles of a digraph is defined similarly. We use the same notation, to refer to the closed trails, Eulerian trails, Eulerian circuits and cycles of a digraph. From the definition, it follows that an Eulerian digraph (a digraph with at least one Eulerian trail) is necessarily connected. Given an edge $e\in E\ev{D}$ such that $\psi\ev{e} = \ev{z_1,z_2}$, then the Eulerian trails of $D$ ending at $e$ is defined as:
\begin{align*}
\mathcal{W}^{e}\ev{D} &= \{ \mathbf{w} = \ev{ v_0, e_1, v_1,\ldots, v_{ d-2 }, e_{ d-1 }, z_1, e, z_2 } : \mathbf{w} \in \mathcal{W}\ev{D} \}
\end{align*}

To define heaps of pieces of a digraph $D$, we take $\mathcal{B}\ev{D}$, the set of cycles, as the set of pieces. Given two cycles $\beta_1,\beta_2\in \mathcal{B}\ev{D}$, we say that $\beta_1$ and $\beta_2$ are \textit{concurrent}, denoted $\beta_1 \mathcal{R} \beta_2$, if they share a vertex: $V\ev{\beta_1} \cap V\ev{\beta_2} \neq \emptyset$. Equivalently, $G_{\mathcal{B}\ev{D},\mathcal{R}}$ is the ``intersection graph" of $\mathcal{B}\ev{D}$ (c.f. \cite[p.~1]{mckee}). We lay out an example for demonstration: 
\begin{example}
\label{ex:digraph_TD}
Consider the digraph $D = \ev{ V\ev{D} = [4], E\ev{D} = \{ e_1,e_2,f_1,f_2,g_1,g_2,h_1,h_2\}, \psi: E\ev{D}\rightarrow V\ev{D} }$, where $\psi$ is the mapping: 
$$
e_1 \mapsto \ev{2,1}, \ e_2 \mapsto \ev{1,2}, \ f_1 \mapsto \ev{1,3}, \ f_2 \mapsto \ev{3,1}, \ g_1 \mapsto \ev{3,2}, \ g_2 \mapsto \ev{2,3}, 
h_1 \mapsto \ev{4,3}, \ h_2 \mapsto \ev{3,4}
$$
A drawing of $D$ is found in \Cref{fig:digraph_TD}. 
\end{example}

\begin{figure}[h!] 
\centering
\includegraphics[width=2in]{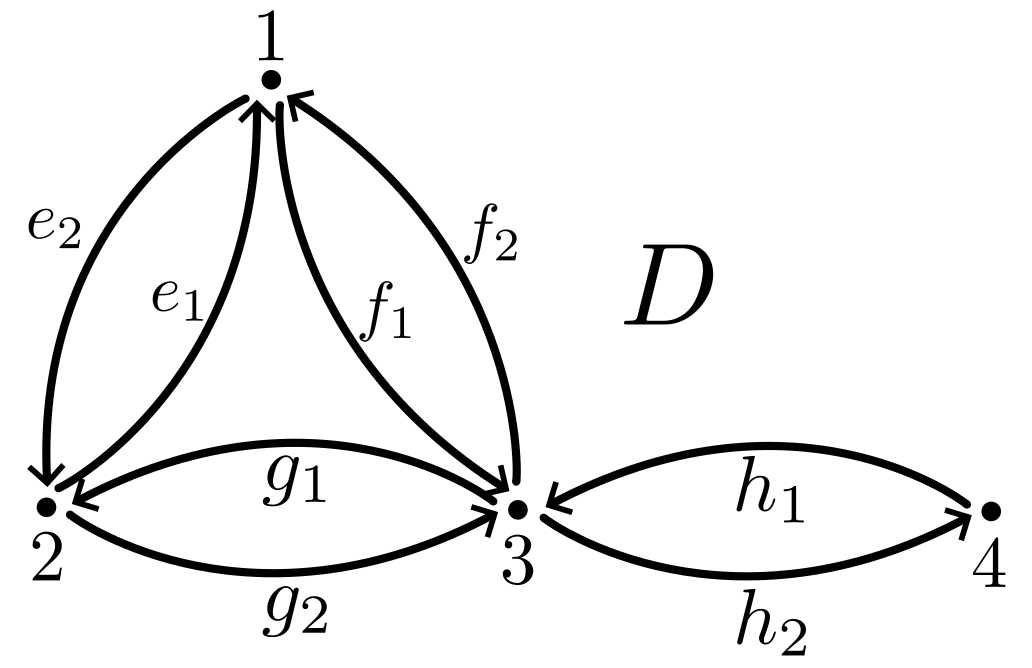}
\caption{The digraph $D$ defined in \Cref{ex:digraph_TD}.}
\label{fig:digraph_TD}
\end{figure}
The digraph $D$ in \Cref{ex:digraph_TD} has $6$ pieces: $\mathcal{B}\ev{D} = \{ \beta_i\}_{i=1}^{6}$, defined and drawn in \Cref{tbl:digraph_TD}. Among the pieces, $\beta_1$ and $\beta_4$ are non-concurrent, and every other pair are concurrent. 
\begin{table} 
\centering
$
\begin{array}{|c|c|c|}
\hline
\raisebox{0.2cm}{\includegraphics[width=4.5cm]{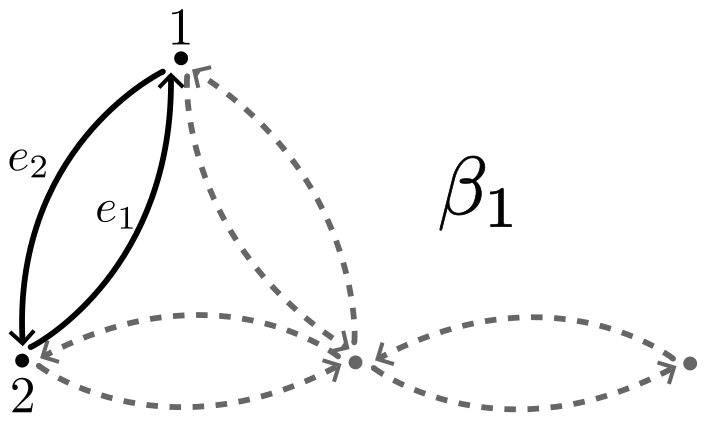}} \phantom{\rule{0.4pt}{2.5cm}} & \raisebox{0.2cm}{\includegraphics[width=4.5cm]{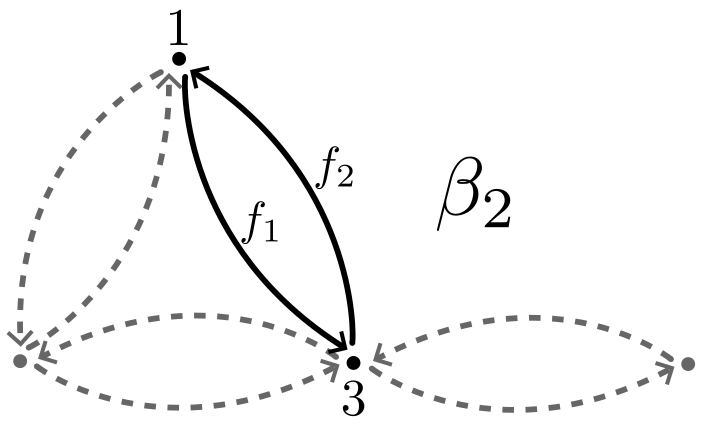}} &
\raisebox{0.2cm}{\includegraphics[width=4.5cm]{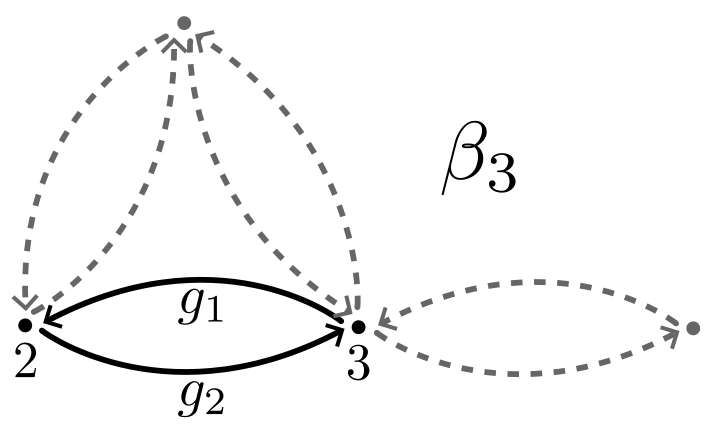}} \\\hline
\includegraphics[width=4.5cm]{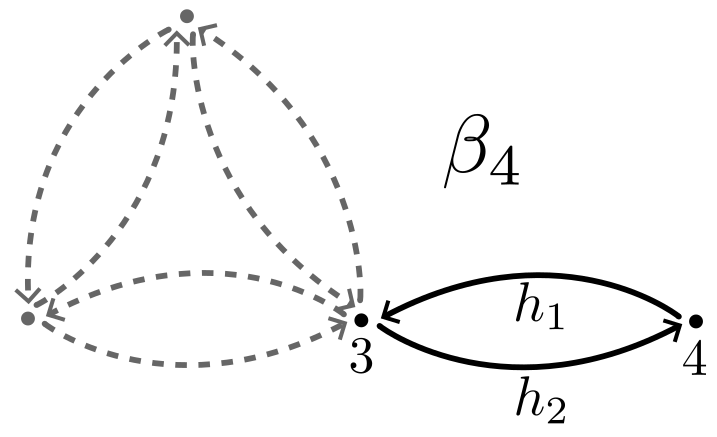} \phantom{\rule{0.4pt}{2.5cm}} & 
\includegraphics[width=4.5cm]{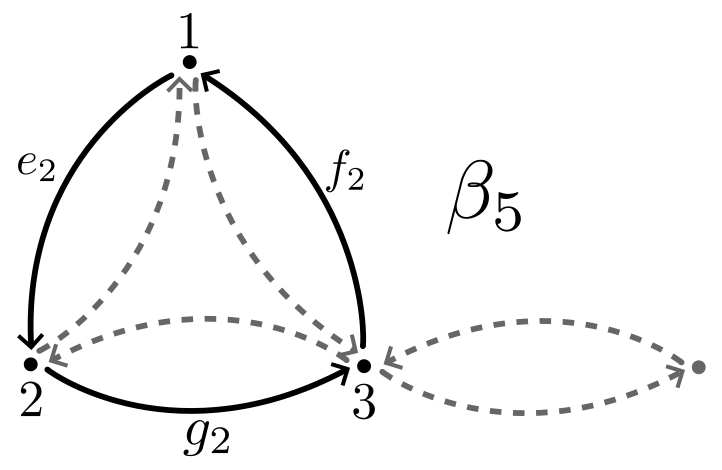} &
\includegraphics[width=4.5cm]{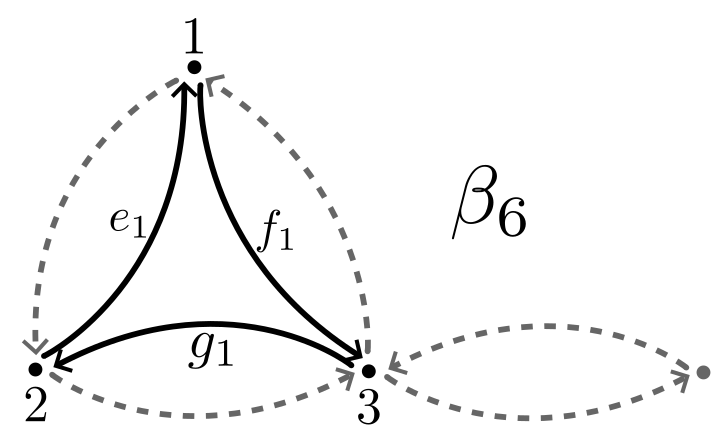} \\\hline
\end{array}
$
\caption{The $6$ pieces of the digraph $D$ defined in \Cref{ex:digraph_TD}.}
\label{tbl:digraph_TD}
\end{table}

From each closed trail $\mathbf{w} \in \mathcal{T}\ev{D}$, we construct a unique sequence of cycles, as follows: First, we express $\mathbf{w}$ uniquely as a concatenation $\alpha \mathbf{w}_1 \gamma$ where,
\begin{itemize}
\item[i)] $\alpha$ is a path starting at $u$ (possibly of length zero); and 
\item[ii)] $\mathbf{w}_1$ is the simple closed trail, the internal vertices of which are distinct and disjoint from $\alpha$. 
\end{itemize}
We say that $\mathbf{w}_1$ is the \textsl{first simple closed subtrail} of $\mathbf{w}$ and $\alpha \gamma$ is the \textsl{remainder}. Note that the remainder $\alpha\gamma$ is again a closed trail at $u$, or a single vertex, i.e., a trail of length $0$. If $\alpha \gamma$ has length $>0$, we may again extract the first simple closed subtrail of $\alpha\gamma$, say $\mathbf{w}_2$. In this way, we iteratively decompose $\mathbf{w}$ into simple closed trails, $\{ \mathbf{w}_i \}_{i=1}^{m}$. Then, the \textit{cycle sequence} of $\mathbf{w}$ is the sequence of cycles $(\beta_1, \ldots, \beta_m )$, where $\beta_i := [\mathbf{w}_i]_\rho$, for each $i=1,\ldots,m$. (We refer the reader to \cite{co2} for a formal definition.)

This process yields a function $ \textup{cs}: \mathcal{T}\ev{D} \rightarrow \mathcal{B}\ev{D}^{*} $, where $\mathcal{B}\ev{D}^{*}$ is the Kleene star closure of $\mathcal{B}\ev{D}$, which maps each closed trail to its cycle sequence:
$$
\mathbf{w}\mapsto \begin{cases} \emptyset & \text{ if } |\mathbf{w}| = 0 \\ (\beta_1, \beta_2,\ldots, \beta_{m}) & \text{ otherwise }
\end{cases}
$$

The counterpart of the operation of decomposing a closed trail into a sequence of cycles, is the operation of insertion, that obtains a new trail from two trails, provided they satisfy certain conditions: 
\begin{definition}
\label{def:insertion_and_concatenation}
\phantom{a}

\begin{enumerate}
\item Given a closed trail $\mathbf{w} = \ev{ v_0, e_1,v_1,\ldots, v_{m-1},e_{m-1},v_m} $ and two indices $i \leq j $, then we define the subtrail,
$$
v_i \mathbf{w} v_j = \ev{ v_i, e_i, v_{i+1},\ldots, v_{j-1},e_{j-1},v_j}
$$
\item Given two not-necessarily-closed trails $\mathbf{w}_1,\mathbf{w}_2$ with $E\ev{ \mathbf{w}_1 } \cap E\ev{ \mathbf{w}_2 } = \emptyset$, assume that the terminal vertex of $\mathbf{w}_1$ is the same as the initial vertex of $\mathbf{w}_2$:
$$ \mathbf{w}_1 = \ev{ v_0, e_1,v_1,\ldots, v_{m-1},e_{m-1},v_m} \text{ and } \mathbf{w}_2 = \ev{ v_0', e_0', v_{1}', \ldots, v_{r-1}',e_{r-1}',v_r'}$$
where, $v_m = v_0'$ and $0\leq m, r$. Then, the \textit{concatenation} of $\mathbf{w}_1$ and $\mathbf{w}_2$, denoted $\mathbf{w}_1 \mathbf{w}_2$, is the trail:
$$ \ev{ v_0, e_1,v_1,\ldots, v_{m-1},e_{m-1},v_m, e_0', v_{1}', \ldots, v_{r-1}',e_{r-1}',v_r'} $$
\item Assume that $\mathbf{w}_1, \mathbf{w}_2$ are closed trails at $v_0$ and $v_0'$, such that $E\ev{ \mathbf{w}_1 } \cap E\ev{ \mathbf{w}_2 } = \emptyset$, $v_0' = v_j \in V\ev{\mathbf{w}_1}$ for some $j\geq 0$ and $V\ev{ \mathbf{w}_1 } \cap \{ v_i \}_{i=0}^{j-1} = \emptyset$. Then, the \textit{insertion} of $\mathbf{w}_1$ into $\mathbf{w}_2$ is defined as the concatenation:
$$
\mathbf{w}_1\cdot \mathbf{w}_2 := (v_0 \mathbf{w}_2 v_j) \ev{ v_0' \mathbf{w}_1 v_r' }  \ev{ v_j \mathbf{w}_2 v_d }
$$ 
In particular, when $j=0$, i.e., both $\mathbf{w}_1$ and $\mathbf{w}_2$ are closed trails at the same vertex $v_0 = v_0'$, then $\mathbf{w}_1 \cdot \mathbf{w}_2 = \mathbf{w}_1 \mathbf{w}_2$ is just the concatenation of $\mathbf{w}_1$ and $\mathbf{w}_2$.
\end{enumerate}
\end{definition}

\subsection{Cancellation Property of the Partition Lattice of an Eulerian Digraph}
In this section, we show one way of proving the cancellation property (\ref{eqn:cancellation}). 
\begin{definition}
Let $D$ be an Eulerian digraph. We define the set of partitions of $D$ into $k\geq 1$ circuits:
$$
\mathfrak{C}_k\ev{D}=\{ \{ [\mathbf{w}_1]_\rho,\ldots, [\mathbf{w}_k]_\rho \} : \mathbf{w}_i \in \mathcal{T}\ev{D} \text{ is a closed trail, for each $i$, and } [ \mathbf{w}_1 \cdots \mathbf{w}_k ]_\rho \in  \mathfrak{C}\ev{D}\}
$$
where $\cdot$ denotes the operation of insertion, as defined in \Cref{def:insertion_and_concatenation} above. Let $f_k\ev{D} := |\mathfrak{C}_k\ev{D}|$, for $k\geq 1$.
\end{definition}
In particular, note that $f_1\ev{D}=|\mathfrak{C}_1\ev{D}|= |\mathfrak{C}\ev{D}|$. Also, there are natural bijections which yield, $|\mathfrak{C}\ev{D}| = |\mathcal{W}^{e}\ev{D}|$, for each edge $e\in E\ev{D}$, and $|\mathfrak{C}\ev{D}| \cdot \textup{deg}_D^{+}\ev{u} = |\mathcal{W}^{u}\ev{D}|$ for each vertex $u\in V\ev{D}$.

\begin{definition}
\label{def:number_cycles}
Consider the partition lattice $\Pi\ev{E\ev{D}}$ on the edge-set of $D$. 
\begin{enumerate}
\item Define the function, 
$$ 
F\ev{ \{ A_1,\ldots,A_k \} } = \prod_{i=1}^{k} - |\mathfrak{C}\ev{ A_i }| = \ev{-1}^{k} \prod_{i=1}^{k} |\mathfrak{C}\ev{ A_i }|
$$
where $|\mathfrak{C}\ev{ A }|$ is the number of Eulerian circuits of each subset $A \subseteq E\ev{D}$. 
\item Let $ G\ev{ b } = \sum_{ a \leq b }  F\ev{a} $ for each $b \in \Pi\ev{E\ev{D}}$. 
\item Let $\mymathbb{1}_D = \mymathbb{1}_{ \Pi\ev{E\ev{ D } } } $ be the unique maximal element, i.e., $\mymathbb{1}_D = \{ E\ev{ D } \} $.
\end{enumerate}
\end{definition}
Observe that:
$$f_k\ev{D} = \sum_{ \substack{ a = \{ A_1,\ldots, A_k \} \leq \mathbbm{1}_D \\ |a| = k } } \prod_{i=1}^{k} f_1\ev{A_i} = \sum_{ \substack{ a \leq \mathbbm{1}_D \\ |a| = k } } \ev{-1}^{k} F\ev{a}  = \ev{-1}^{k} \sum_{ \substack{ a \leq \mathbbm{1}_D \\ |a| = k } }  F\ev{a} $$

\begin{definition}
\label{def:induced_sublattice} 
Given an Eulerian digraph $D$, let $T\ev{D} = \ev{ \mathcal{T}, \leq\big\vert_\mathcal{T} }$ be the sub-poset of the partition lattice $\Pi\ev{ E\ev{D} }$, induced by $\mathcal{T} := \{ b\in \Pi\ev{ E\ev{D} } : F\ev{b} \neq 0 \}$, i.e., the set of partitions of $E\ev{D}$ into Eulerian parts. (Note that $F\ev{b} \neq 0$ if and only if each part of $b = \{A_1,\ldots,A_m\}$ is the edgeset of an Eulerian digraph.)
\end{definition}

By definition, the unique maximal element of $\Pi\ev{E\ev{D}}$ belongs to $T\ev{D}$ and is denoted by $\mathbbm{1}_D$. On the other hand, minimal elements of $T\ev{D}$ are exactly the partitions of $E\ev{D}$ into cycles. Unless $D$ has a unique partition into directed cycles, $T\ev{D}$ does not have a unique minimal element. Let $Q\ev{D} := \min\ev{T\ev{D}}$ be the set of partitions of $D$ into cycles. 

We will show that $T\ev{D}$ is closed under $\vee_{\Pi\ev{E\ev{D}}}$:
\begin{lemma}
\label{lem:unique_min_lattice}
The sub-poset $\ev{T\ev{D}, \vee_{\Pi\ev{E\ev{D}}}}$ is a join-semilattice. Furthermore, $T\ev{D}$ has a unique minimal element if and only if $\ev{T\ev{D}, \vee_{\Pi\ev{E\ev{D}}}, \wedge_{\Pi\ev{E\ev{D}}}}$ is a lattice. 
\end{lemma}
\begin{proof} 
By an easy induction one can see that:
\begin{equation}
\label{eqn:eulerian_iff}
\text{If a digraph $N$ is connected and $E\ev{N}$ is a disjoint union of Eulerian digraphs, then $N$ is Eulerian.}
\end{equation}

Let $a = \{A_1,\ldots,A_t\},b=\{B_1,\ldots,B_m\}\in T\ev{D}$ be given. We claim that $c:= a \vee_{\Pi\ev{E\ev{D}}} b = \left\{ \bigcup_{C \in \mathcal{C} } C  \ : \  \mathcal{C} \in (a\cup b) / \overline{\varspade} \right\} \in T\ev{D}$. Given any part $C_1 \cup \ldots \cup C_r$ of $c$, where $C_i \in a\cup b$ and $C_1 \varspade \ldots \varspade C_r$, note that $C_1 \cup \ldots \cup C_r$ is connected, as we have $C_i \cap C_{i+1} \neq \emptyset$, for each $i$, by the definition of $\varspade$. On the other hand, since $a$ refines $c$, each part of $c$ is a disjoint union of some parts of $a$. As every part of $a$ is an Eulerian digraph, it follows that each part of $c$ satisfies (\ref{eqn:eulerian_iff}), and so, $F\ev{c} \neq 0$ and $c\in T\ev{D}$. 

Next, assume that $c\in T\ev{D}$ is the unique minimal element. In particular, each part of $c$ is a cycle. Given any two elements $a,b \in T\ev{D} $, we claim that $a \wedge_{\Pi\ev{E\ev{D}}} b = \{ A \cap B : A\in a, \ B\in b,\  A \varspade B\} \in T\ev{D}$. Since $c$ refines $a$, any part $A\in a$ is a disjoint union of cycles of $c$. Similarly for any part of $b$. Therefore, any part $A\cap B$ of $a\wedge b$ is both connected (by the definition of $\varspade$) and is a disjoint union of cycles of $c$, and so $F\ev{A\cap B} \neq 0$, by (\ref{eqn:eulerian_iff}). Conversely, if $\ev{T\ev{D}, \vee_{\Pi\ev{E\ev{D}}}, \wedge_{\Pi\ev{E\ev{D}}}}$ is a lattice, then it does not have more than one minimal element, by the definition of a lattice.
\end{proof}

For demonstration, we draw the Hasse diagram of $T\ev{D}$ in \Cref{fig:induced_join_semilattice}, where the values $F\ev{b}$ are noted next to each element $b\in T\ev{D}$.
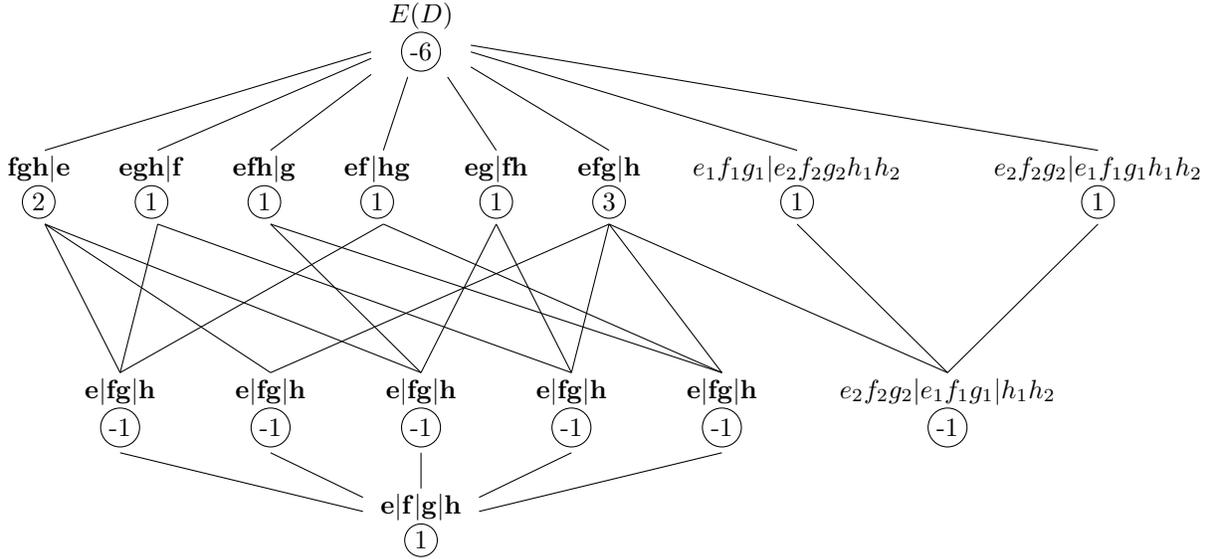
\begin{figure}[ht!]
\centering
\begin{tikzpicture}[node distance=1.5cm, every node/.style={inner sep=1pt, outer sep=1pt}] 
\node (n0) at (-2,1) {\Large{ $\substack{E\ev{D} \\ \circled{-6} } $ } };
\node (n1) at (-7,-1) {\Large{$\substack{\mathbf{f} \mathbf{g} \mathbf{h} |  \mathbf{e}  \\ \circled{2} } $ } };
\node (n2) at (-5.5,-1) {\Large{$ \substack{ \mathbf{e} \mathbf{g} \mathbf{h} |  \mathbf{f} \\ \circled{1} } $ } };
\node (n3) at (-4,-1) {\Large{$ \substack{  \mathbf{e} \mathbf{f} \mathbf{h} | \mathbf{g} \\ \circled{1} }  $ } };
\node (n4) at (-2.5,-1) {\Large{$ \substack{ \mathbf{e} \mathbf{f} | \mathbf{h} \mathbf{g} \\ \circled{1} } $ } };
\node (n5) at (-1,-1)  {\Large{ $ \substack{  \mathbf{e} \mathbf{g} | \mathbf{f} \mathbf{h} \\ \circled{1} }  $ } };
\node (n6) at (0.5,-1)  {\Large{ $\substack{  \mathbf{e} \mathbf{f} \mathbf{g} | \mathbf{h} \\ \circled{3} }  $ } };
\node (n7) at (3,-1)  {\Large{ $  \substack{  e_1 f_1 g_1 | e_2 f_2 g_2  h_1 h_2 \\ \circled{1} } $ } };
\node (n8) at (7,-1)  {\Large{ $ \substack{  e_2 f_2 g_2 | e_1 f_1 g_1  h_1 h_2 \\ \circled{1} } $ } };
\node (n9) at (-6,-4)  {\Large{ $ \substack{  \mathbf{e} | \mathbf{f} \mathbf{g} |  \mathbf{h} \\ \circled{-1} }  $ } };
\node (n10) at (-4,-4) {\Large{ $ \substack{  \mathbf{e} | \mathbf{f} \mathbf{g} |  \mathbf{h}  \\ \circled{-1}  } $ } };
\node (n11) at (-2,-4) {\Large{  $\substack{  \mathbf{e} | \mathbf{f} \mathbf{g} |  \mathbf{h}  \\ \circled{-1}  }  $ } };
\node (n12) at (0,-4) {\Large{ $\substack{  \mathbf{e} | \mathbf{f} \mathbf{g} |  \mathbf{h}  \\ \circled{-1} }   $ } };
\node (n13) at (2,-4) {\Large{ $ \substack{  \mathbf{e} | \mathbf{f} \mathbf{g} |  \mathbf{h}  \\ \circled{-1} } $  } };
\node (n14) at (5,-4) {\Large{ $ \substack{  e_2 f_2 g_2 | e_1 f_1 g_1 | h_1 h_2  \\ \circled{-1} }  $ } };
\node (n15) at (-2,-5.5) {\Large{ $ \substack{  \mathbf{e} | \mathbf{f} | \mathbf{g} | \mathbf{h}   \\ \circled{1} }  $ } };
\draw (n0) -- (n1.north); \draw (n0) -- (n2.north); \draw (n0) -- (n3.north); \draw (n0) -- (n4.north); \draw (n0) -- (n5.north); \draw (n0) -- (n6.north); \draw (n0) -- (n7.north); \draw (n0) -- (n8.north);
\draw (n1.south) -- (n9.north); \draw (n1.south) -- (n10.north); \draw (n1.south) -- (n11.north);
\draw (n2.south) -- (n9.north); \draw (n2.south) -- (n12.north); 
\draw (n3.south) -- (n11.north); \draw (n3.south) -- (n13.north); 
\draw (n4.south) -- (n9.north); \draw (n4.south) -- (n13.north);
\draw (n5.south) -- (n11.north); \draw (n5.south) -- (n12.north);
\draw (n6.south) -- (n10.north); \draw (n6.south) -- (n12.north);  \draw (n6.south) -- (n13.north);  \draw (n6.south) -- (n14.north); 
\draw (n7.south) -- (n14.north);  \draw (n8.south) -- (n14.north); 
\draw (n9.south) -- (n15); \draw (n10.south) -- (n15); \draw (n11.south) -- (n15); \draw (n12.south) -- (n15); \draw (n13.south) -- (n15); 
\end{tikzpicture}
\caption{The join-semilattice $T\ev{D}$ induced by partitions into Eulerian digraphs, where $\mathbf{e}$ denotes $e_1 e_2$ for each $e_1,e_2\in E\ev{D}$, and curly brackets are omitted. The numbers are the values $F\ev{b}$, for each element $b\in T\ev{D}$.}
\label{fig:induced_join_semilattice}
\end{figure}

\begin{remark}
\label{rmk:theta_distinction}
Given an edge $e\in E\ev{D}$, then each walk $\mathbf{w}\in \mathcal{W}^{e}\ev{ D }$ ending at $e$ has a cycle sequence $\textup{cs}\ev{ \mathbf{w} } = \ev{ \beta_1, \ldots, \beta_k}$ (for $k\geq 1$) such that $\theta\ev{ \ev{ \beta_1, \ldots, \beta_k} } = \{ \beta_1, \ldots, \beta_k \}$ is a minimal element of $T\ev{ D }$, where $\theta$ is the forgetful mapping of \Cref{def:orientations}. For each partition $a\in Q\ev{D}$ into cycles, define 
$$\mathcal{W}_a^{e}\ev{ D } := \{ \textbf{w} \in \mathcal{W}^{e}\ev{ D } :  \theta\ev{ \textup{cs}\ev{ \mathbf{w} } } = a \}$$
which implies, 
$$ \mathcal{W}^{e}\ev{D} = \bigsqcup_{a\in Q\ev{D} } \mathcal{W}_a^{e}\ev{ D } $$
\end{remark}

To express the M\"{o}bius function of the join-semilattice $T\ev{D}$, we will need the notion of a ``decomposition pyramid" from \cite{co2}, defined next:
\begin{definition}
\phantom{a}
\begin{enumerate}
\item For a piece $\beta \in \mathcal{B}\ev{D}$, let $E\ev{\beta}$ be the edge set of $\beta$. A pyramid $P \in \mathbbm{p}\ev{D}$ is called a \textit{decomposition pyramid}, provided 
\begin{equation}
\label{eqn:decomposition}
E\ev{D} = \bigsqcup_{\beta \in P} E\ev{\beta}
\end{equation}
is a partition of $E\ev{D}$ into cycles. Let $\mymathbb{dp}\ev{ D }$ be the set of decomposition pyramids of a connected digraph $D$.
\item For a fixed edge $e\in E\ev{D}$, the set of decomposition pyramids with maximal piece containing $e$ is denoted $\mymathbb{dp}^{e}\ev{ \mathcal{B}\ev{D}, \mathcal{R}}$, or succinctly $\mymathbb{dp}^{e}\ev{ D }$.
\end{enumerate}	

\end{definition}
\begin{lemma}
\label{lem:pyramids_and_walks_containing_edge}
Let $D$ be a Eulerian digraph. Assume that $D$ has a partition $a = \{ \beta_1, \ldots, \beta_k \}$ into cycles. Let $\beta\in a$ be a fixed cycle and let $e\in \beta$ be a fixed edge on the piece $\beta$.
\begin{enumerate}
\item A pyramid of $a$ is full if and only if it is a decomposition pyramid: $\mymathbb{P}\ev{a,\mathcal{R}} = \mymathbb{dp}\ev{a,\mathcal{R}}$. 
\item The trails of $D$ with cycle sequence from $a$ and ending at $e$ are in bijection with full pyramids with maximal piece $\beta$:
$$ |\mathcal{W}_a^{e}\ev{ D }|  = |\mathbb{P}^{\beta}\ev{ a ,\mathcal{R} } | $$
\end{enumerate}
\end{lemma}
\begin{proof}

\phantom{a}

\begin{enumerate}
\item Given a full pyramid $P \in \mymathbb{P}\ev{a}$, then we have $E\ev{P} = E\ev{D}$, and so $P$ is a decomposition pyramid. Conversely, if $P$ is a decomposition pyramid, then $E\ev{D} = E\ev{P}$, and it follows by (\ref{eqn:decomposition}), that every cycle of $a$ must appear in $P$ and so, $P$ is a full pyramid. 
\item By \cite[Theorem 2.3, p.~12]{co2}, there is a bijection 
$$ 
\bigsqcup_{a\in Q\ev{D} } \mathcal{W}_a^{e}\ev{ D } = \mathcal{W}^{e}\ev{D} \underset{g}{\rightarrow} \mymathbb{dp}^{e}\ev{ D } = \bigsqcup_{a\in Q\ev{D} } \mymathbb{dp}^{e}\ev{ a }
$$

By the definition of $\mathcal{W}_a^{e}\ev{D}$, $\mymathbb{dp}^{e}\ev{ a }$ and $g$, the restriction $\mathcal{W}_a^{e}\ev{D} \rightarrow \mymathbb{dp}^{e}\ev{ a  }$ is also bijective. As $\beta$ is the unique piece containing $e$, we infer that,
$$
| \mathcal{W}_{a}^{e}\ev{D} | =  | \mymathbb{dp}^{e}\ev{ a } |  = | \mymathbb{dp}^{\beta}\ev{ a } | = |\mathbb{P}^{\beta}\ev{ a } | 
$$
where the last step follows from Part 1.
\end{enumerate}
\end{proof}

\begin{lemma}
\label{lem:up_set_isom_conn}
Let $D$ be a connected Eulerian digraph. Then, for each minimal element $a \in T\ev{D}$, the up-set $\upset[T\ev{D}]{a}$ is a lattice, where the meet and join operations are the restrictions of $\wedge_{\Pi\ev{D}}$ and $\vee_{\Pi\ev{D}}$ to $\upset[T\ev{D}]{a}$. Furthermore, we have a lattice isomorphism between $\upset[T\ev{D}]{a}$ and the bond lattice of $a$:
$$ 
\upset[T\ev{D}]{a} \rightarrow L\ev{ a }  
$$
\end{lemma}
\begin{proof}
We define maps $\phi$ and $\psi$ which are inverses of each other:
\begin{align*}
\phi: \upset[T\ev{D}]{a} & \rightarrow L\ev{ a }\\
\{ D_1,\ldots,D_k \} & \mapsto \{ B_1,\ldots, B_k \} \\ 
& \text{ where $ B_i = \{ \beta \in a: E\ev{\beta} \subseteq E\ev{D_i} \}$} \\ 
\psi: L\ev{ a } & \rightarrow \upset[T\ev{D}]{a} \\
\{ B_1,\ldots, B_k \}	& \mapsto \lc \bigcup_{\beta\in B_1} E\ev{\beta}, \ldots, \bigcup_{\beta\in B_k} E\ev{\beta} \rc
\end{align*}
It is routine to check that $\phi$ and $\psi$ also respect the partial order relations $\leq_{T\ev{D}}$ and $\leq_{L\ev{a}}$.
\end{proof}

\begin{theorem}
\label{thm:mobius_inv_cycles}
Let $D$ be a connected Eulerian digraph that is not a cycle. Then, we have
$$G\ev{ \mathbbm{1}_D } = \sum_{ a \leq \mathbbm{1}_D }  F\ev{a} = 0 $$
\end{theorem}
\begin{proof}
Fix an edge $e\in E\ev{D}$. For each minimal element $a\in Q\ev{D}$, let $\beta_a$ be the unique piece of $a$, containing $e$. By \Cref{rmk:theta_distinction}, we have:
$$
|\mathfrak{C}\ev{D}| = |\mathcal{W}^{e}\ev{D}| = \sum_{a\in Q\ev{D} } |\mathcal{W}_a^{e}\ev{ D }|
$$
For a minimal element $a\in Q\ev{D}$, we have $G\ev{ a } = F\ev{ a } = \ev{ -1 }^{|a|}$. Also, by \Cref{lem:up_set_isom_conn}, the up-set of $a$ is isomorphic to the bond lattice, $\upsetdisplay[T\ev{ D }]{a} \cong L\ev{ a }$, and so by \Cref{lem:balanced_pieces_mobius}, the M\"{o}bius function satisfies $\mu_{\upset[T\ev{D}]{a} }\ew{ a, \mymathbb{1}_{T\ev{D}} } = \mu_{ L\ev{ a } }\ew{ \mymathbb{0}_{L\ev{a}}, \mymathbb{1}_{L\ev{a}} }  = \ev{ -1 }^{ 1 - |a| } | \mathbb{P}^{ \beta_a } \ev{ a, \mathcal{R} } |$.

We apply induction on the number of edges $|E\ev{D}|$, to show that $G\ev{ \mathbbm{1}_D } = 0$. For the basis step, let $D$ be a connected Eulerian digraph with $|E\ev{D} | = 2$. Then, the statement vacuously holds, as the only connected Eulerian digraph on two edges is the single directed cycle. For the inductive step, let $|E\ev{D} | \geq 3$ be fixed. 

Consider any non-minimal element $b = \{B_1,\ldots,B_k\}\in T\ev{ D }$ with $b <_{T\ev{D}} \mymathbb{1}_D$. In particular, $b$ has at least $ k\geq 2 $ parts, each part defining an Eulerian digraph. Also, $b$ has at least one component, say $B_{i_0}$ (for $1\leq i_0 \leq k$), which is not a cycle and satisfies $|B_{i_0}| < |E\ev{D}|$. By the inductive hypothesis, we obtain $G\ev{ \mymathbb{1}_{B_{i_0}} } = 0$. On the other hand, for each $c \leq_{T\ev{D}} b$, there is a partition of $c := c_1 \sqcup \ldots \sqcup c_k$ such that $c_i = \{ C_{i 1},\ldots, C_{i r_i}\} $ and $B_i = C_{i 1} \sqcup \ldots \sqcup C_{i r_i}$. In other words, we have $c_i \in T\ev{B_i}$, for each $i=1,\ldots,k$. It follows that,
\begin{align*}
& G\ev{ b } = \sum_{ c\leq_{T\ev{D}} b} F\ev{ c } = \sum_{ \substack{c\leq_{T\ev{D}} b \\ c = c_1 \cup \ldots \cup c_k, \\ c_i \in T\ev{B_i} } } \prod_{i = 1}^{ k } \prod_{j = 1}^{ r_i } - | \mathfrak{C}\ev{ C_{ij} } |  \\
& = \prod_{i=1}^{ k } \sum_{ x = \{X_1,\ldots,X_m\} \leq \mymathbb{1}_{B_i} } \prod_{j = 1}^{ m } - | \mathfrak{C}\ev{ X_j } |  \\
& = \prod_{i=1}^{ k } \sum_{ x \leq_{T\ev{B_i}} \mymathbb{1}_{B_i} } F\ev{x} = \prod_{i=1}^{ k } G\ev{ \mymathbb{1}_{B_i}  } = 0 & \text{ by the inductive hypothesis, since $G\ev{\mymathbb{1}_{B_{i_0}}}  = 0$ }. 
\end{align*} 

Therefore,
\begin{flalign*}
& -  | \mathfrak{C}\ev{ D } | =  F\ev{ \mathbbm{1} } = \sum_{ b \in T\ev{D} } \mu\ev{ b, \mathbbm{1} } G\ev{b} &&\hspace{-7cm}\text{ by M\"{o}bius inversion }\\ 	
&  =   \mu\ev{ \mathbbm{1}, \mathbbm{1} } G\ev{ \mathbbm{1} } + \sum_{ a \in Q\ev{D} } \mu\ev{ a, \mathbbm{1} } G\ev{a} + \sum_{ \substack{ b \notin Q\ev{D} \\ b\neq \mymathbb{1}} } \mu\ev{ b, \mathbbm{1} } G\ev{b}  \\ 
&  = G\ev{ \mathbbm{1} }  +  \sum_{ a \in Q\ev{D} } \mu_{ T\ev{D} }\ev{ a, \mathbbm{1} } G\ev{a} &&  \\ 
& \hspace{0cm}\text{ by the inductive hypothesis, $G\ev{b} = 0$, for each $b\in T\ev{D} \setminus Q\ev{D}$ such that $b\neq \mymathbb{1}$, as shown above } && \\ 
&  = G\ev{ \mathbbm{1} }  +  \sum_{ a \in Q\ev{D} } \mu_{ \upset[T\ev{D}]{a} }\ev{ a, \mathbbm{1} } G\ev{a} &&\hspace{-7cm}\text{ by \Cref{lem:mu_up_set}} 
\end{flalign*}
\begin{flalign*}
&  = G\ev{ \mathbbm{1} }  +  \sum_{ a \in Q\ev{D} } \mu_{ L\ev{a} }\ev{ a, \mathbbm{1} } G\ev{a} &&\hspace{-7cm}\text{ by \Cref{lem:up_set_isom_conn}} \\ 
& = G\ev{ \mathbbm{1} }  +  \sum_{ a \in Q\ev{D} } \ev{ -1 }^{ 1 - |a| } \cdot | \mathbb{P}^{ \beta_a } \ev{ a, \mathcal{R} } | \cdot \ev{ -1 }^{|a|}  && \\ 
& \hspace{0cm} \text{ by \Cref{lem:balanced_pieces_mobius} and the equation $G\ev{ a } = \ev{ -1 }^{|a|}$ } &&  \\ 
& =  G\ev{ \mathbbm{1} }  - \sum_{ a \in Q\ev{D} } | \mathcal{W}_a^{ e  }\ev{ D } |  && \hspace{-7cm}\text{ by \Cref{lem:pyramids_and_walks_containing_edge}}\\ 
& = G\ev{ \mathbbm{1} }  - | \mathfrak{C}\ev{ D } |    
\end{flalign*}
which implies that $G\ev{ \mathbbm{1} } =0$. 
\end{proof}

As a corollary, we have:
\begin{corollary}
\label{cor:indecomp_zero}
Let $D$ be a (connected) Eulerian digraph that is not a cycle. Then, 
$$ \sum_{k\geq 1} \ev{ -1 }^{k} f_k\ev{D} = 0 $$
\end{corollary}
\begin{proof}
$$
f_k\ev{D} = \sum_{ \substack{ b = \{ B_1,\ldots, B_k \} \leq_{T\ev{D}} \mathbbm{1}_D \\ |b| = k } } \prod_{i=1}^{k} | \mathfrak{C}\ev{ B_i } | = \sum_{ \substack{ b \leq_{T\ev{D}}\mathbbm{1}_D \\ |b| = k } } \ev{-1}^{k} F\ev{b}  = \ev{-1}^{k} \sum_{ \substack{ b \leq_{T\ev{D}}\mathbbm{1}_D \\ |b| = k } }  F\ev{b} 
$$
Therefore, we obtain 
$$ \sum_{k\geq 1} \ev{ -1 }^{k} f_k\ev{D} =  \sum_{k\geq 1} \sum_{ \substack{ b \leq_{T\ev{D}}\mathbbm{1}_D \\ |b| = k } } F\ev{b} = \sum_{ b\leq_{T\ev{D}}\mathbbm{1}_D } F\ev{b} = G\ev{ \mymathbb{1}_{D} } = 0$$
\end{proof}

\subsection{The Martin Polynomial}
\label{sec:martin_poly}
The definition below is found in \cite{bollobas,lasvergnas, martin}:
\begin{definition}
Given a connected digraph $D$,
\label{def:martin_circuit_partition}
\begin{enumerate}
\item The \textit{circuit partition polynomial} of $D$ is 
$$
r_{D}\ev{t} := \sum_{k\geq 1} t^{k} f_k\ev{D} 
$$
\item The \textit{Martin polynomial} of $D$ is
$$
s_D\ev{t} := \dfrac{1}{t-1} r_D\ev{t-1} =  \sum_{k\geq 1} \ev{t-1}^{k-1} f_{k}\ev{D} 
$$
\end{enumerate}
\end{definition}
Given a (connected) Eulerian digraph with maximum out-degree $\Delta \geq 2$, then it is shown in \cite{lasvergnas} that $s_D\ev{t} $ is divisible by $\prod_{i=2}^{\Delta} (t + \Delta - i) = \ev{ t + \Delta - 2} \cdots \ev{t+1} \cdot t$. In particular, unless $D$ is a single directed cycle, we obtain 
$$
s_D\ev{-\ev{\Delta-2}}=\ldots = s_D\ev{0} = 0 
$$
and in particular, (\ref{eqn:cancellation}) follows. Evaluations of $t$ at other integer values also yield interesting results (see \cite{lasvergnas,bollobas}):
\begin{enumerate}
\item $r_D\ev{0}=0$ and $s_D\ev{1} = f_1\ev{D} = |\mathfrak{C}\ev{D}| $ 
\item $r_D\ev{1} = s_D\ev{2} = \sum_{k\geq 1} f_k\ev{D} = \prod_{v\in V\ev{D}} \textup{deg}_{D}^{+}\ev{v}! $
\item For a fixed integer $k>0$, see \cite[Corollary 2, p.~265]{bollobas} for an expression of $r\ev{-k}$ in terms of cycles of $D$. 
\end{enumerate}

We derive an identity connecting the Martin polynomial (and the circuit partition polynomial) to the bond lattices of intersection graphs of partitions of $D$ into cycles. 
\begin{proposition}
\label{prop:hopefully_original}
Let $D$ be a (connected) Eulerian digraph. Then
$$
s_D\ev{1-t} = - \sum_{a\in Q\ev{D}} \ev{-1}^{|a|} \chi_{L\ev{a}}\ev{t} 
$$
where $\chi_{L\ev{a}}$ is the characteristic polynomial of the bond lattice $L\ev{a}$ (q.v.~\Cref{def:posets_prelim}).
\end{proposition}
\begin{proof}
For each $b = \{ D_1,\ldots,D_k \}\in T\ev{D}$, we define maps $\phi$ and $\psi$ between $\{ a\in Q\ev{D} : a\leq_{T\ev{D}} b \}$ and $Q\ev{D_1}\times \cdots \times Q\ev{D_k}$ which are inverses of each other (compare with \Cref{lem:up_set_isom_conn}):
\begin{align*}
\phi & : a \mapsto \ev{ a_1,\ldots, a_k } \text{ where $ a_i = \{ \beta \in a: E\ev{\beta} \subseteq E\ev{D_i} \} $ }  \\
\psi & : \ev{c_1,\ldots,c_k} \mapsto c_1 \sqcup \ldots \sqcup c_k 
\end{align*}
Let $\triangleleft$ be any linear order on $E\ev{D}$. For each $B\subseteq E\ev{D}$, let $e_B$ be the minimal edge with respect to $\triangleleft$. Furthermore, for each Eulerian $B\subseteq E\ev{D}$ and each partition $c\in Q\ev{B}$, let $\beta\ev{c}$ be the unique cycle of $c$ containing $e_B$. Then,  
\begin{align*}
& s_D\ev{1-t} = \sum_{k\geq 1} \ev{-t}^{k-1} f_k\ev{D}   = - \sum_{k\geq 1} \ev{-t}^{k} \sum_{b = \{ D_1,\ldots,D_k \}\in T\ev{D}} \lav \mathfrak{C}\ev{D_1} \times \cdots \times \mathfrak{C}\ev{D_k} \rav \\ 
& = \sum_{k\geq 1}  \ev{-t}^{k-1}  \sum_{b = \{ D_1,\ldots,D_k \}\in T\ev{D}} \lav \mathcal{W}^{e_{D_1}} \ev{D_1}  \times \cdots \times \mathcal{W}^{e_{D_k}} \ev{D_k} \rav \\ 
& = \sum_{k\geq 1}  \ev{-t}^{k-1}  \sum_{b = \{ D_1,\ldots,D_k \}\in T\ev{D}} \sum_{\ev{c_1,\ldots,c_k} \in Q\ev{D_1}\times \cdots \times Q\ev{D_k}} \lav \mathcal{W}^{e_{D_1}}_{c_1} \ev{D_1} \times \cdots \times \mathcal{W}^{e_{D_k}}_{c_k} \ev{D_k} \rav \\
& \hspace{2cm}\text{ by \Cref{rmk:theta_distinction} } \\ 
& = \sum_{k\geq 1}  \ev{-t}^{k-1}  \sum_{b = \{ D_1,\ldots,D_k \}\in T\ev{D}} \sum_{ \substack{a\in Q\ev{D} \\  b \in \upset[T\ev{D}]{a} } } \lav  \mathcal{W}^{e_{D_1}}_{a_1} \ev{D_1} \times \cdots \times \mathcal{W}^{e_{D_k}}_{a_k} \ev{D_k} \rav \\
& = \sum_{k\geq 1}  \ev{-t}^{k-1}  \sum_{ a\in Q\ev{D}  }  \sum_{  b = \{ D_1,\ldots,D_k \}\in \upset[T\ev{D}]{a} } \lav  \mathcal{W}^{e_{D_1}}_{a_1} \ev{D_1} \times \cdots \times \mathcal{W}^{e_{D_k}}_{a_k} \ev{D_k} \rav \\
& = \sum_{k\geq 1} \ev{-t}^{k-1} \sum_{ a\in Q\ev{D}  } \sum_{  b = \{B_1,\ldots,B_k\} \in L\ev{a} }  \ev{-1}^{1-|B_1|}\mu_{L\ev{B_1}}\ev{ \mymathbb{0}_{B_1}, \mymathbb{1}_{B_1} } \cdots  \ev{-1}^{1-|B_k|}\mu_{L\ev{B_k}}\ev{ \mymathbb{0}_{B_k}, \mymathbb{1}_{B_k} }    \\
&\hspace{2cm} \text{by \Cref{lem:balanced_pieces_mobius}, \Cref{lem:up_set_isom_conn} and \Cref{lem:pyramids_and_walks_containing_edge}} \\ 
& = \sum_{ a\in Q\ev{D}  } \sum_{k\geq 1}  \ev{-t}^{k-1}  \sum_{  b = \{B_1,\ldots,B_k\} \in L\ev{a} }  \ev{-1}^{k+|B_1|+\ldots+|B_k|}\mu_{ \downset[L\ev{a}]{b}}\ev{ \mymathbb{0}_{a} , b}  \\
&\hspace{2cm} \text{ by \Cref{lem:heap_lattice_product_poset_isom}, as $\mymathbb{0}_{B_1} \cup \ldots \cup \mymathbb{0}_{B_k} = \mymathbb{0}_{a} $, $\mymathbb{1}_{B_1} \cup \ldots \cup \mymathbb{1}_{B_k} = b$,} \\ 
& \hspace{2cm} \text{and the M\"{o}bius function is multiplicative over products of posets (c.f.~\cite[p.~344]{vanlint})} \\ 
& = -\sum_{ a\in Q\ev{D}  }  \sum_{  b \in L\ev{a} } t^{|b|-1}   \ev{-1}^{|a| } \mu_{L\ev{a}}\ev{ \mymathbb{0}_{a}, b}  \\
& \hspace{2cm} \text{by \Cref{lem:mu_down_set}}\\
&  = - \sum_{ a\in Q\ev{D}  }  \sum_{  b \in L\ev{a} } \ev{-1}^{|a|} t^{\textup{rk}\ev{L\ev{a}} - \textup{rk}\ev{b}} \mu\ev{ \mymathbb{0}_{a}, b} =  - \sum_{ a\in Q\ev{D}  } \ev{-1}^{|a|} \chi_{L\ev{a}}\ev{t} 
\end{align*}
where $\textup{rk}$ is the rank function of the bond lattice $L\ev{a}$ (q.v.~\Cref{def:posets_prelim}).
\end{proof}
It is well-known (\cite[p.~162]{stanley_enumerative}) that $t\cdot \chi_{L\ev{a}}\ev{t} = P_{ G_a }\ev{t}$, where $P_{ G_a }\ev{t}$ is the chromatic polynomial of the intersection graph $G_a$. It follows that 
$$
r_D\ev{-t} = \ev{-t} \cdot s_D\ev{1-t} = \sum_{a\in Q\ev{D}} \ev{-1}^{|a|} P_{ G_a }\ev{t}
$$
\begin{example}
Let $D$ be the digraph from \Cref{ex:digraph_TD}. Then, the partitions into cycles are
$$
a_1 = \{ \beta_1,\beta_2,\beta_3,\beta_4\} \qquad \text{ and } \qquad a_2 = \{ \beta_4, \beta_5,\beta_6\}
$$
It follows that
\begin{align*}
& s_{D}\ev{t} = 6 + 11 \ev{ t-1 } + 6 \ev{t-1}^{2} + \ev{t-1}^{3} \\
& \chi_{L\ev{a_1}}\ev{t} = -4 +8t -5t^2 + t^3 \\ 
& \chi_{L\ev{a_2}}\ev{t} = 2 - 3t + t^2 \\ 
& - s_{D}\ev{1-t} = -\ev{ 6 - 11 t + 6 t^{2} - t^{3} } = \ev{-1}^{4}\ev{ -4 +8t -5t^2 + t^3 } + \ev{-1}^{3} \ev{ 2 - 3t + t^2 }
\end{align*}

\end{example}

For more results on the Martin polynomial, we refer the reader to \cite{martin_remarkable,jaeger, monaghan_new_results, bollobas,lasvergnas, martin}. For connections between the Martin polynomial and the Tutte polynomial, c.f.~\cite{monaghan}.

We end this part with a question: The Martin polynomial of a multi-graph is defined in \cite{lasvergnas}. Is there an identity similar to that of \Cref{prop:hopefully_original} which relates the Martin polynomial of a multi-graph $X$ to the chromatic polynomials of the intersection graphs of the partitions of $X$ into cycles?

\section{Deduction of the Harary-Sachs Theorem for Graphs} 
\label{sec:deduction}
\subsection{Preliminaries on Infragraphs}
We start with preliminaries on the characteristic polynomial of a hypergraph. Let $\mathcal{H}$ be a simple hypergraph, of rank $k$ (every edge consisting of $k\geq 2$ vertices) and order $|V\ev{\mathcal{H}} | = n$. The (normalized) \textit{adjacency hypermatrix} $\mathbb{A}_\mathcal{H} = \ev{a_{i_1 \ldots i_k}}$ of $\mathcal{H}$ is the $n$-dimensional symmetric $k$-hypermatrix, where $a_{i_1,\ldots,i_k} = 1/(k-1)!$ if $\{i_1,\ldots,i_k\}\in E\ev{\mathcal{H}}$, and $0$ otherwise. The \textit{Lagrangian} of a hypermatrix $\mathbb{A} = \ev{ a_{i_1,\ldots,i_k} }$ is the $n$-variate polynomial
$$
F_\mathbb{A}\ev{\textbf{x}} = \sum_{i_1 \ldots i_k = 1 }^{n} a_{i_1 \ldots i_k} x_{i_1} \cdots x_{i_k}
$$
The \textit{characteristic polynomial} of $\mathcal{H}$, denoted $\chi_{\mathcal{H}}\ev{t}$, is the $n$-variate resultant of 
$$
\nabla \ev{ t\cdot F_\mathcal{I}\ev{\textbf{x}}  - F_\mathcal{H}\ev{\textbf{x}}  }
$$
where $\nabla: \C[x_1,\ldots,x_n] \rightarrow \prod_{j=1}^{n} \C[x_1,\ldots,x_n]$ is the $n$-variate gradient and $\mathcal{I}$ is the order-$k$ and dimension-$n$ identity hypermatrix, with ones on the diagonal and zeroes elsewhere (cf.~\cite{cd1,cc1,cc2,co2}).

Given a multi-hypergraph $X$, then we may obtain a simple hypergraph $\underline{X}$, called the \textit{flattening} of $X$, by removing duplicate edges. More formally,

\begin{definition}
\label{def:flattening}
\phantom{a}

Given a multi-hypergraph $X = \ev{ V, E, \varphi: E\rightarrow \binom{V}{k} }$, then the \textit{flattening} $\underline{X}$ of $X$ is the simple hypergraph $\underline{X} = \ev{ V, E /\unsim, \widetilde{ \varphi } : E/ \unsim \rightarrow \binom{V}{k} }$, obtained by taking the quotient of the edge-set, under the equivalence relation of parallelism (q.v.~\Cref{def:multi_hyper_graph}), where $\widetilde{\varphi}([e]_{\unsim}) = \varphi(e)$ (which is well-defined, since $\varphi(e)=\varphi(f)$ for all $f \in [e]_{\unsim}$). 
\end{definition}

\begin{definition}[Vertex-fixing isomorphisms of multi-graphs and digraphs]
\label{def:approx}
Given two undirected multi-hypergraphs $X = \ev{ V, E, \varphi}$ and $X' = \ev{V', E', \varphi'}$, then we write $X \approx X'$ if $V = V'$ and $ m_X\ev{ \{  v_1,\ldots,v_k \} } = m_{X'}\ev{ \{ v_1,\ldots,v_k \} }$ for each $v_1,\ldots,v_k\in V = V'$ (q.v.~\Cref{def:multi_hyper_graph}). Equivalently, define $X = \ev{V,E,\phi} \approx X' = \ev{ V', E', \phi' }$ if and only if $ V = V'$ and there exists a bijection $\eta: E \rightarrow E'$, that preserves the parallelism relation $\unsim$: For each $e_1, e_2 \in E$, we have $e_1 \unsim_{X} e_2$ if and only if $\eta\ev{e_1} \unsim_{X'} \eta\ev{e_2}$. 

Similarly, for two digraphs $D = \ev{ V, E, \psi}$ and $D' = \ev{ V' , E', \psi'}$, we write $ D \approx D' $, provided $V = V'$ and $ m_D\ev{ u, v } = m_{D'}\ev{ u, v}$ for each $u,v\in V = V'$. Given any $\mathbf{O} \in \mathcal{O}\ev{X} / \approx$, we define $m_\mathbf{O}\ev{ u, v } := m_O\ev{u,v}$, for any $u,v \in V\ev{X} $ and any representative $O\in \mathbf{O}$. 


\end{definition}

In \cite{cc1,co2}, the characteristic polynomial $\chi_{\mathcal{H}}\ev{t}$ is expressed in terms of the counts of occurrences of ``Veblen'' hypergraphs in $\mathcal{H}$, which is defined next:

\begin{definition}

\label{def:veblen_multi_hypergraph}

\phantom{a}
Let $k\geq 2$ be fixed. 

\begin{enumerate}
\item[\textit{i)}] A multi-hypergraph $X$ is \textit{$k$-valent} if the degree of each vertex of $X$ is divisible by $k$. 
\item[\textit{ii)}] A $k$-valent multi-hypergraph of rank $k$ is called a \textit{Veblen hypergraph}. 
\end{enumerate}

For a Veblen hypergraph $X$, let $c\ev{X}$ be the number of connected components of $X$. Let $\mathcal{V}^{m}$ denote the set of isomorphism classes of Veblen hypergraphs with $m$ components, where $m\geq 1$. In particular, $\mathcal{V}^{1}$ is the set of isomorphism classes of connected Veblen hypergraphs. Let $\mathcal{V}^{\infty}$ be the set isomorphism classes of (possibly disconnected) Veblen hypergraphs. In other words, we define
$$ \mathcal{V}^{\infty} = \bigsqcup_{ m = 1 }^{\infty} \mathcal{V}^{m} $$

A $k$-valent multi-hypergraph $X$ is an \textit{infragraph} of a simple hypergraph $\mathcal{H}$, provided $\underline{X}$ is a subgraph of $\mathcal{H}$. Let $ \textup{Inf}^{m}\ev{\mathcal{H}}$ denote the set of infragraphs of $\mathcal{H}$ with $m\geq 1$ components, up to equivalence under $\approx$ (see \Cref{def:approx}). In other words, $\textup{Inf}^{m}\ev{\mathcal{H}}$ is the set of equivalence classes $[X]_{ \approx }$ of infragraphs $X$ with $m$ components. By an abuse of notation, we will refer to the elements $[X]_{\approx} \in \textup{Inf}^{m}\ev{\mathcal{H}}$ by $X$, omitting the brackets. In particular, $\textup{Inf}^{1}\ev{\mathcal{H}}$ is the set of isomorphism classes of connected infragraphs of $\mathcal{H}$. We define,
$$
\textup{Inf}\ev{\mathcal{H}} = \bigsqcup_{ m = 1 }^{\infty} \textup{Inf}^{m}\ev{\mathcal{H}}
$$
\end{definition}

A connected Veblen hypergraph $X$ is \textit{decomposable}, if it has a proper non-trivial Veblen subgraph $Y$. In this case, $E\ev{X}\setminus E\ev{Y}$ also induces a Veblen hypergraph. Inductively, we can obtain at least one decomposition of $X$ into indecomposable Veblen hypergraphs. However, the decomposition of $X$ into indecomposables is not necessarily unique. We define the set of partitions (decompositions) of a Veblen hypergraph $X$ into connected Veblen hypergraphs: 
\begin{definition}
\label{def:parallel_equiv}
Given a Veblen hypergraph $X = \ev{ V, E, \varphi} \in \mathcal{V} $, define
$$\mathcal{S}\ev{X} := \left\{ \{  A_1, \ldots, A_m \} : \emptyset \neq A_i \subseteq E \text{ and } A_i \in \mathcal{V}^{1} \text{ and } E = \bigsqcup_{i=1}^{m} A_i \text{ is a disjoint union}  \right\} $$
\end{definition}

Given the equivalence relation $\unsim$ on $E$ denoting parallelism of edges, we will define further equivalence relations $\unsim_1$ and $\unsim_2$ on $\mathcal{P}(E)$ and $\mathcal{P}(\mathcal{P}(E))$: 

\begin{definition}

\phantom{a}

\begin{enumerate}
\item[\textit{i)}] Let $\text{Sym}_{\unsim} \ev{ E } $ be the subgroup of $\text{Sym}\ev{E}$, consisting of permutations of $E$, preserving the parallelism relation $\unsim$:
$$ \text{Sym}_{\unsim} \ev{ E }  := \{ \sigma\in \text{Sym}\ev{E}: \sigma\ev{ e } \unsim e \text{ for each } e\in E \}$$
\item[\textit{ii)}] Denote by $\unsim_1$ an equivalence relation on the power set $\mathcal{P}\ev{E}$, as follows: Given $A,B \subseteq E$, 
$$ 
A \ \unsim_1 \ B \text{ if and only if } \exists \sigma \in \text{Sym}_{\unsim} \ev{ E } \text{ such that } \sigma\ev{A} = B
$$
Furthermore, for each $S = \{ A_1,\ldots, A_m \} \in \mathcal{S}\ev{X}$, we have $S\subseteq \mathcal{P}\ev{E}$ and so, the following cardinality is well-defined:
$$ 
\alpha_{ S } := \prod_{j=1}^{r} | [A_{i_j}]_{\unsim_1} | 
$$
where the distinct elements of $S/\unsim_1$ are $ [A_{i_1}]_{\unsim_1}, \ldots,  [A_{i_r}]_{\unsim_1} $ for some $1\leq r \leq m$.
\item[\textit{iii)}] Define an equivalence relation $\unsim_2$ on $\mathcal{S}\ev{X} \subseteq \mathcal{P}(\mathcal{P}(E))$, as follows: Given $S, T \in \mathcal{S}\ev{X}$, 
$$ 
S \unsim_2 T \text{ if and only if } \exists \sigma \in \text{Sym}_{\unsim} \ev{ E } \text{ so that } \sigma(S) = T
$$
\end{enumerate}
\end{definition}
\begin{remark}
More generally, given a set $E$ and a fixed equivalence relation $\unsim$ on $E$, we can define an equivalence relation $\unsim_i$ on the $i$th power set $\mathcal{P}^{i}\ev{E}$ of a set $E$, for $i\geq 1$: Let $\mathcal{P}^{0}\ev{E} = E$. For $i\geq 1$, let $\mathcal{P}^{i}\ev{E} = \mathcal{P}(\mathcal{P}^{i-1}\ev{E})$. Then, for each $A,B \in \mathcal{P}^{i}\ev{E}$, let $A \unsim_i B$ if and only if $\sigma \ev{ A } = B $ for some $\sigma \in \text{Sym}_{\unsim} \ev{ E }$. (We only need $\unsim_i$ for $i=1,2$, for present purposes.)
\end{remark}
Henceforth, by an abuse of notation, we will write $\unsim$ for each of the above equivalence relations. This does not generate any confusion, because the ground set of each equivalence class is different. We succinctly write $\widetilde{\mathcal{S}}\ev{X}$ instead of $\mathcal{S}\ev{X} / \unsim$. 

The product of factorials of the multiplicities of the edges of $X$ is denoted by,
$$
M_X = \prod_{ [e]_{\unsim} \in E\ev{X} / \unsim } m_X\ev{e} ! 
$$
For any $\mathbf{S} = \{ A_1, \ldots, A_m\} \in \mathcal{S}\ev{X}$, we put $ M_S = \prod_{i=1}^{m} M_{A_i} $. Also, for any $\mathbf{S}\in \widetilde{\mathcal{S}}\ev{X}$, we define $M_\mathbf{S} := M_S$ for any $S\in \mathbf{S}$, which is well-defined, independent of the choice of representative. Note that, for any $\mathbf{S}\in \widetilde{\mathcal{S}}\ev{ X }$, we have 
$$ 
| \mathbf{S} | = \frac{M_X}{ M_\mathbf{S} \alpha_{ \mathbf{S} } } 
$$
where $\alpha_{ \mathbf{S} } := \alpha_{ S }$, for any choice of representative $S\in \mathbf{S}$. 

We need the following definitions for a digraph $D$:
\begin{definition}
Let $D$ be a digraph and $u\in V\ev{D}$ be a fixed vertex. 
\begin{enumerate}
\item Define $N_D := \prod_{u\in V\ev{D}} \textup{deg}_{D}^{+}\ev{u} !$ (q.v.~\Cref{sec:martin_poly})
\item Let $K_D\ev{u} := \prod_{v\in V\ev{D}} m_D\ev{u,v} ! $ and $K_D := \prod_{u\in V\ev{D}} K_D\ev{u}$
\end{enumerate}
\end{definition}

\begin{remark}
\label{rmk:approx_remark}
Let $X=\ev{ V\ev{X}, E\ev{X}, \phi}$ be a connected multi-graph of rank $2$ and let $O \in \mathcal{O}\ev{X} $ be an orientation.
\begin{enumerate}
\item For an edge $e\in E\ev{X}$ with $\phi\ev{e} = \{ u,v\}$, we have $m_X\ev{e} = m_O\ev{u,v} + m_O\ev{v,u}$. 
\item 
$$ 
| [O]_{\approx} | =  \prod_{ \{u,v\} \in \binom{V(X)}{2} } \dfrac{ m_X\ev{ \{ u, v\} }! }{  m_{ O } \ev{ u, v }!  m_{ O } \ev{ v, u }!  }  = \dfrac{ \prod_{ \{u,v\} \in \binom{V(X)}{2} } m_X\ev{ \{ u, v\} }  }{ \prod_{u \in V\ev{D}} \prod_{v \in V\ev{D}} m_{ O } \ev{ u, v }! } = \dfrac{ M_X }{ K_O }
$$
\end{enumerate}
\end{remark}

For the expression of the Harary-Sachs Theorem for hypergraphs, we will need the definition of a ``rooting".
\begin{definition}
\label{def:rooting}
Let $X = \ev{ [n], E, \varphi }$ be a Veblen hypergraph. Given an edge $e\in E\ev{X}$, for each $u\in e$, we define the \textit{$u$-rooted directed star of $e$}, denoted $S_u\ev{ e }$, as the digraph $T$ with $V\ev{T} = \varphi(e)$ and $E\ev{T} = \{ \ev{ u,v} : v\in \varphi(e) \text{ and } v\neq u\}$. Given $X\in \mathcal{V}^{\infty}_d$, a \textit{rooting} is a $d$-tuple of stars $\ev{ S_{v_1}\ev{ e_1 },\ldots, S_{v_d}\ev{ e_d } }$ such that
\begin{itemize}
\item[\textit{i)}] $(\forall i \in [d-1]) (v_i\leq v_{i+1})$ 
\item[\textit{ii)}] $e_1,\ldots,e_d$ are pairwise distinct and $E = \{ e_1,\ldots,e_d\}$. 
\item[\textit{iii)}] $[n]=\{v_1,\ldots,v_d\}$
\item[\textit{iv)}] $(\forall i \in [d]) (v_i \in e_i)$
\end{itemize}
\end{definition}

From each rooting $R = \ev{ S_{v_1}\ev{ e_1 },\ldots, S_{v_d}\ev{ e_d } }$ of Veblen hypergraph $X$, we obtain a digraph by taking the union of the edges of stars, adding multiplicities when two edges are parallel. Formally, let $D_R$ be the digraph with $V(D_R) = V(X)$, $E(D_R) = \{(v,w,j): v=v_j, w \in e_j \setminus \{v\}\}$, and $\psi(v,w,j)=(v,w)$ for each $v,w \in [n]$. We say that $R$ is an \textit{Eulerian rooting} if $D_R$ is Eulerian. Let $ \mathfrak{R}\ev{X} $ be the set of Eulerian rootings of $X$. For a rooting $R\in \mathfrak{R}\ev{X} $, let $i$ be chosen as the root of $r_i$-many stars, $S_{i}\ev{e_{i1}},\ldots,S_{i}\ev{e_{ir_i}}$. Then, the rooting $R$ has the form 
$$
\ev{S_{1}\ev{e_{11}},\ldots,S_{1}\ev{e_{1r_1}},\ldots, S_{n}\ev{e_{n1}},\ldots,S_{n}\ev{e_{nr_n}}} 
$$
By the definition of a rooting, it follows that for each pair of parallel edges $e,e'$ and each vertex $i\in [n]$, we have $S_{i}\ev{e} = S_{i}\ev{e'}$. We will define an equivalence relation on the set of rootings, where two rootings are considered equivalent, if one is obtained from the other by interchanging two stars with the same root (and possibly non-parallel edges).

\begin{definition}
\label{def:equiv}
Let $R=\ev{S_{v_1}\ev{e_1},\ldots,S_{v_d}\ev{e_d}}$ be any rooting. Define an equivalence relation $\equiv$ on $\mathfrak{R}\ev{ X }$, as the transitive closure of the relations,
\begin{align*}
\ev{S_{v_1}\ev{e_1},\ldots, S_{v_i}\ev{e_i}, S_{v_{i+1}}\ev{e_{i+1}} \ldots ,S_{v_d}\ev{e_d}} & \equiv \ev{S_{v_1}\ev{e_1},\ldots, S_{v_{i+1}}\ev{e_{i+1}},  S_{v_i}\ev{e_i}, \ldots ,S_{v_d}\ev{e_d}} 
\end{align*}
where $v_i = v_{i+1}$ and $1 \leq i \leq d-1$.
\end{definition}

Let $\mathfrak{R}\ev{X} / \!\equiv$ be the set of equivalence classes of rootings of $X$ under $\equiv$. Instead of $\mathfrak{R}\ev{X}/ \!\equiv$, we will succinctly write $ \overline{\mathfrak{R}}\ev{X}$. Note that $R\equiv R'$ implies $D_R \approx D_{R'}$. Hence, the following are well-defined, independent of the choice of a representative:
$$
D_\mathbf{R} := [D_R]_{\approx} \text{ and } \mathfrak{C}\ev{ D_\mathbf{R} } := \mathfrak{C}\ev{ D_R } \text{ for any $R \in \mathbf{R}$.}
$$

\begin{remark}
\label{rmk:rootings_and_orientations}
In the special case of a Veblen multi-graph $X$ of rank $k=2$, 
\begin{enumerate}
\item $ |\mathbf{R}| = \dfrac{ \prod_{u\in V\ev{D_\mathbf{R}} } \textup{deg}^{+}\ev{u}! }{ \prod_{u\in V\ev{D_\mathbf{R}} } \prod_{v\in V\ev{D_\mathbf{R}} } m\ev{u,v}! } = \dfrac{N_{D_{\mathbf{R}}}}{ K_{D_\mathbf{R}}}$
\item It is easy to see that $R\equiv R'$ if and only if $D_R \approx D_{R'}$, for two Eulerian rootings $R, R' \in \mathfrak{R}\ev{X}$. (This is not necessarily true for rank $k\geq 3$.) In other words, choosing $\mathbf{R} \in \overline{\mathfrak{R}}\ev{ X }$ is equivalent to choosing an Eulerian orientation of $X$, up to $\approx$. Let $ \widehat{\mathcal{O}}\ev{X} := \{ O \in \mathcal{O}\ev{X} :  | \mathfrak{C}\ev{ O } | \geq 1  \} $ be the set of Eulerian orientations of $X$. Then, we have a natural bijection $ \{ D_\mathbf{R}: \mathbf{R} \in \overline{\mathfrak{R}}\ev{X} \} \leftrightarrow \widehat{\mathcal{O}}\ev{X} / \approx$. 
\end{enumerate}
\end{remark}

We are ready to state the Harary-Sachs Theorem for hypergraphs. One version of this theorem is found in \cite{co2}:
\begin{theorem}
\label{thm:harary_sachs_hyper}
Let $\mathcal{H}$ be a hypergraph, of order $k\geq 2$ and dimension $|V\ev{\mathcal{H}}| = n$. Let $N := \textup{deg}\ev{\chi_\mathcal{H}\ev{t}}$ be the degree of the characteristic polynomial. Then, we have 
$$
t^{-N} \cdot \chi_\mathcal{H}\ev{t} = \sum_{X\in \textup{Inf}\ev{\mathcal{H}} } \ev{-1}^{c\ev{X}} t^{- |E\ev{X}|  }  w_n\ev{X}
$$
where: 
\begin{enumerate}
\item[(i)] $c\ev{X}$ is the number of connected components of $X$.
\item[(ii)] $w_n\ev{X} = - \sum_{ \mathbf{S} \in \widetilde{\mathcal{S}} \ev{ X }} \ev{ -\ev{k-1}^{n} }^{ c\ev{ \mathbf{S} } } \dfrac{1}{\alpha_{\mathbf{S}}} C_{ \mathbf{S} }$
\item[(iii)] $C_\mathbf{S}$ denotes the \textit{associated coefficient}, defined in \cite{cd1, cc1} and extended multiplicatively to partitions of a Veblen hypergraph. In particular, for a connected Veblen hypergraph $G$, we define:
$$
C_G = \sum_{\mathbf{R} \in  \overline{\mathfrak{R}}\ev{G} } | \mathbf{R} | \cdot \dfrac{ | \mathfrak{C}\ev{D_\mathbf{R}} | }{ N_{D_{\mathbf{R}}} } 
$$
For a partition $S = \{A_1,\ldots,A_m\} \in \mathcal{S}\ev{X}$ of a Veblen hypergraph $X$, we define $ C_{ S } = \prod_{i=1}^{m} C_{A_{i}} $. Given $\mathbf{S} \in \widetilde{\mathcal{S}}\ev{X}$, then we define $C_{\mathbf{S}} := C_S$ where $S\in \mathbf{S}$ is any chosen representative.
\end{enumerate}
\end{theorem}
\begin{remark}
The weight function in \cite{co2} incorporates the factor $t^{-|E\ev{X}|}$, whereas the way it is defined above does not. This is inconsequential for our present purposes.
\end{remark}

\subsection{Deduction of Harary-Sachs Theorem for 2-Graphs}
As shown in \cite{harary} by Harary and in \cite{sachs} by Sachs, the coefficients of a simple graph $\mathcal{G}$ can be calculated in terms of the counts of elementary graphs in $\mathcal{G}$:
\begin{theorem}[Harary-Sachs Theorem for graphs of rank $2$]
\label{thm:harary_sachs_ordinary}
Let $\mathcal{G}$ be a simple graph. Call a Veblen hypergraph $X \in \mathcal{V}^{\infty}$ \textup{elementary} if it is a disjoint union of edges and cycles. Let $\mathcal{E}$ be the set of isomorphism classes of elementary Veblen hypergraphs. Then, we have the following formula for the characteristic polynomial $\chi_\mathcal{G} \ev{ t }$ (note that $\textup{deg} \ev{\chi_\mathcal{G}\ev{t}} = |V\ev{\mathcal{G}}|$):
$$t^{-|V\ev{\mathcal{G}}|} \cdot \chi_\mathcal{G} \ev{ t } = \sum_{ X \in \mathcal{E} }  \ev{-1}^{c\ev{X}} 2^{z\ev{X}} \genfrac{\lbrack}{\rbrack}{0pt}{0}{\mathcal{G}}{X} $$
where
\begin{enumerate}
\item $c\ev{X}$ is the number of components of an elementary graph $X$
\item $z\ev{X}$ is the number of cycles in $X$ of length $\geq 3$
\item $\genfrac{\lbrack}{\rbrack}{0pt}{0}{\mathcal{G}}{X} $ is the number of isomorphic copies of $X$ in $G$
\end{enumerate} 
\end{theorem}

We aim to deduce \Cref{thm:harary_sachs_ordinary}, the Harary-Sachs Theorem for 2-Graphs, from \Cref{thm:harary_sachs_hyper}. We first introduce some preliminary definitions and lemmas. The following lemma is from \cite[p.~7]{cc2}, deriving a simpler formula for the associated coefficient, in the case of a multi-graph of rank $k=2$:
\begin{lemma}
\label{lem:orientations_count}
For a Veblen multi-graph $X$ of rank $k=2$, we have
$$ 
C_X = \dfrac{ | \mathfrak{C}\ev{ X } | }{ M_X } 
$$
\end{lemma}
\begin{proof}
Each Eulerian circuit of $X$ gives rise to an orientation of $X$, which we can group by equivalence under $\approx$. Therefore,
\begin{align*}
& \dfrac{1}{M_X} | \mathfrak{C} \ev{ X } | = \dfrac{1}{M_X} \sum_{O \in \widehat{\mathcal{O}}\ev{X} } | \mathfrak{C} \ev{ O } |  = \dfrac{1}{M_X}\sum_{ \mathbf{R} \in \overline{\mathfrak{R}}\ev{ X } } \sum_{\substack{ O \in \widehat{\mathcal{O}}\ev{X} \\ O \approx D_\mathbf{ R } } }  | \mathfrak{C} \ev{ O } | \\
&  = \dfrac{1}{M_X}\sum_{ \mathbf{R} \in \overline{\mathfrak{R}}\ev{ X } } |\{ O \in \widehat{\mathcal{O}}\ev{X} : O \approx D_{\mathbf{R}}\}| \cdot | \mathfrak{C}\ev{D_\mathbf{R}} | \\
& \text{ as $| \mathfrak{C}\ev{O} | = |\mathfrak{C}\ev{D_\mathbf{R}}|$ for any orientation $O\approx D_\mathbf{R}$} \\ 
& = \dfrac{1}{M_X}  \sum_{ \mathbf{R} \in \overline{\mathfrak{R}}\ev{ X } } \dfrac{M_X}{K_{D_\mathbf{R}}} \cdot | \mathfrak{C}\ev{D_\mathbf{R}} | && \hspace{-5cm}\text{by \Cref{rmk:approx_remark} } \\ 
& = \sum_{ \mathbf{R} \in \overline{\mathfrak{R}}\ev{ X } } \dfrac{N_{D_\mathbf{R}}}{K_{D_\mathbf{R}}} \cdot \dfrac{ | \mathfrak{C}\ev{D_\mathbf{R}} | }{ N_{D_\mathbf{R}} }  = \sum_{ \mathbf{R} \in \overline{\mathfrak{R}}\ev{ X } } | \mathbf{R} | \cdot \dfrac{ | \mathfrak{C}\ev{D_\mathbf{R}} | }{ N_{D_\mathbf{R}} } = C_X.
\end{align*}
\end{proof}

In what follows, we will be interested in the number of different ways an Eulerian orientation $O \in \widehat{\mathcal{O}}\ev{X}$ of a Veblen hypergraph $X$ may be obtained from partitions of $X$ into sub-multi-graphs. Hence, we define:
$$
\mathcal{S}_t \ev{ O } = \{ \ev{ S, P }: S =  \{ X_1, \ldots, X_t\} \in \mathcal{S}\ev{X}, \ P = \{ O_1, \ldots, O_t \} , \ O_i \in \widehat{\mathcal{O}} \ev{ X_i } ,\ O = \bigsqcup_{i=1}^{t} O_i \} 
$$
In particular, note that $\mathcal{S}_1 \ev{ O } = \{ \ev{ \{ X \}, \{ O \} } \}$. We now argue that there is a correspondence between decompositions of the Eulerian circuits of $O$ into $t\geq 1$ parts and tuples of Eulerian circuits obtained from elements of $\mathcal{S}_t \ev{ O } $, for each fixed Eulerian orientation $ O \in \widehat{\mathcal{O}}\ev{X} $:
\begin{lemma}
\label{lem:orientations_decomp_circuits}
For $t\geq 1$, a multi-graph $X$ of rank $2$ and an Eulerian orientation $O \in \widehat{\mathcal{O}}\ev{X}$, we have
$$ 
| \mathfrak{C}_t\ev{ O } | = \sum_{  \substack{ \ev{ S,P } \in \mathcal{S}_t \ev{ O }  \\ P = \{ O_1, \ldots, O_t \} } } \prod_{i=1}^{t} | \mathfrak{C}\ev{ O_i } | 
$$
\end{lemma}
\begin{proof}
Let $t\geq 1$ be fixed. Define the mapping,
\begin{align*}
f: \bigsqcup_{  \substack{ \ev{ S,P } \in \mathcal{S}_t \ev{ O }  \\ P = \{ O_1, \ldots, O_t \} } }  \mathfrak{C}_1\ev{ O_1 }\times \cdots \times \mathfrak{C}_1\ev{ O_t }  & \rightarrow \mathfrak{C}_t\ev{ O }
\end{align*}
by $f \ev{ z_1,\ldots,z_t } = \{ z_1, \ldots, z_t \} $. This mapping is a bijection, by the definition of $\mathfrak{C}_t\ev{O}$ and $\mathcal{S}_t\ev{O}$.
\end{proof}

Finally, we state and prove the main theorem of this section. 
\begin{theorem}
\label{thm:weight_zero}
For a connected decomposable Veblen multi-graph $X$ of rank $k = 2$, 
$$w_n\ev{X} =0$$
\end{theorem}
\begin{proof}
For any connected $X \in \mathcal{V}^{1} $, if $X$ is a decomposable infragraph, then, clearly, no orientation of $X$ is a directed cycle.
\begin{align*}
& w_n\ev{X} = \sum_{ \mathbf{S} \in \widetilde{\mathcal{S}} \ev{ X } } \ev{ - 1 }^{c\ev{\mathbf{S}}} \dfrac{1}{\alpha_{ \mathbf{S} } } C_\mathbf{S}  =  \sum_{ S \in \mathcal{S} \ev{ X } } \ev{ - 1 }^{c\ev{S}} \dfrac{ M_S }{ M_X } C_S  && \hspace{-1cm} \text{ since $| \mathbf{S} | = \dfrac{M_X}{M_\mathbf{S} \cdot \alpha_\mathbf{S} } $} \\ 
& = \dfrac{1}{ M_X } \sum_{ t\geq 1 } \ev{ -1 }^{t} \sum_{ S =  \{ X_1, \ldots, X_t \} \in \mathcal{S}_t \ev{ X } }  \prod_{ i=1 }^{t} M_{X_i} C_{X_i} && \hspace{-2cm}\text{as $ M_S $ and $C_S$ are multiplicative} \\
& = \dfrac{1}{ M_X } \sum_{ t\geq 1 } \ev{ -1 }^{t} \sum_{ S =  \{ X_1, \ldots, X_t \} \in \mathcal{S}_t \ev{ X } }   \prod_{ i=1 }^{t} |\mathfrak{C}\ev{X_i}| && \hspace{-1cm} \text{ by \Cref{lem:orientations_count}} \\
& = \dfrac{1}{ M_X } \sum_{ t\geq 1 } \ev{ -1 }^{t} \sum_{ S =  \{ X_1, \ldots, X_t \} \in \mathcal{S}_t \ev{ X } }   \prod_{ i=1 }^{t} \sum_{ O \in \widehat{\mathcal{O}}\ev{X_i} } | \mathfrak{C}\ev{ O_i } |  \\
& = \dfrac{1}{ M_X } \sum_{ t\geq 1 } \ev{ -1 }^{t}  \sum_{ O \in \widehat{\mathcal{O}}\ev{X} }  \sum_{  \substack{ \ev{ S,P } \in \mathcal{S}_t \ev{ O }  \\ P = \{ O_1, \ldots, O_t \} } }  \prod_{i=1}^{t} | \mathfrak{C}\ev{ O_i } |  \\
& = \dfrac{1}{ M_X }  \sum_{ O \in \widehat{\mathcal{O}}\ev{X} }  \sum_{ t\geq 1 } \ev{ -1 }^{t} |\mathfrak{C}_{t}\ev{O}|  
\end{align*}
which is equal to zero, by \Cref{cor:indecomp_zero}, provided $X$ is decomposable.
\end{proof}
By \Cref{thm:weight_zero} and \Cref{thm:harary_sachs_hyper}, for a graph $\mathcal{G}$ of rank $2$, the only non-trivial contributions to the characteristic polynomial $\chi_{\mathcal{G}}\ev{t}$ are from indecomposable infragraphs $X \in \textup{Inf}\ev{\mathcal{G}}$. On the other hand, by Veblen's Theorem (\cite[p.~87]{veblen}) those indecomposable infragraphs are exactly the \textit{elementary subgraphs}, i.e., a disjoint union of edges and cycles of length $\geq 3$. Therefore, the classical Harary-Sach Theorem for graphs (\Cref{thm:harary_sachs_ordinary}) follows. 

%

\section{Acknowledgement}
We thank Erik Panzer for pointing out the connection between (\ref{eqn:cancellation}) and the Martin polynomial. 

\bibliography{bibliography.bib}

\end{document}